\documentclass[lebTerpaper,11pt]{amsart}
\usepackage{amsmath,amssymb,amsxtra,amsthm, amstext, amscd,amsfonts,fancyhdr, hyperref,enumerate}
\usepackage{url}

\makeatletter
\DeclareRobustCommand{\NS}{%
  \textsection\@ifnextchar\NS{}{}%
}
\makeatother

\usepackage[OT2,T1]{fontenc}
\DeclareSymbolFont{cyrlebTers}{OT2}{wncyr}{m}{n}
\DeclareMathSymbol{\Sha}{\mathalpha}{cyrlebTers}{"58}

\newcommand{\bA}{{\mathbb{A}}}

\newcommand{\bF}{{\mathbb{F}}}

\newcommand{\bN}{{\mathbb{N}}}

\newcommand{\bP}{{\mathbb{P}}}
\newcommand{\bQ}{{\mathbb{Q}}}
\newcommand{\bR}{{\mathbb{R}}}

\newcommand{\bT}{{\mathbb{T}}}

\newcommand{\bZ}{{\mathbb{Z}}}

\newcommand{\C}{{\mathcal{C}}}
\newcommand{\D}{{\mathcal{D}}}

\newcommand{\R}{{\mathcal{R}}}

\newcommand{\GL}{\operatorname{GL}}

\newcommand{\upchi}{{\raise.35ex\hbox{$\chi$}}}

\makeatletter
\@namedef{subjclassname@2010}{%
	\textup{2010} Mathematics Subject Classification}
\makeatother

\newtheorem{theorem}{Theorem}[section]
\newtheorem{corollary}[theorem]{Corollary}
\newtheorem{proposition}[theorem]{Proposition}
\newtheorem{lemma}[theorem]{Lemma}

\theoremstyle{definition}

\newtheorem{question}[theorem]{Question}

\newtheorem{remark}[theorem]{Remark}

\numberwithin{equation}{section}

\begin{document}
	
	\title[]{Quartic polynomials in two variables do not represent \\ all non-negative integers }
	
	\author{Stanley Yao Xiao}
	\address{Department of Mathematics and Statistics \\
		University of Northern British Columbia \\
		3333 University Way \\
		Prince George, British Columbia, Canada \\  V2N 4Z9}
	\email{StanleyYao.Xiao@unbc.ca}
	\indent
	
\author{Shuntaro Yamagishi}
\address{IST Austria, Am Campus 1, 3400 Klosterneuburg, Austria}
\email{shuntaro.yamagishi@ist.ac.at}
\indent

	
\begin{abstract}
In this paper, we prove that there does not exist $F \in \bQ[x,y]$ of degree $4$ such that $F(\mathbb{Z}^2) = \mathbb{Z}_{\geq 0}$.
In particular, this answers a question by John S. Lew and Bjorn Poonen for quartic polynomials.
\end{abstract}
	
	\maketitle
	
	\section{Introduction}
Representation of integers by polynomials is a fundamental problem in number theory that goes back
at least to the time of Gau{\ss} and Lagrange. Regarded as one of the most important theorems in additive number theory,
Lagrange proved that every non-negative integer is a sum of four squares. Given a polynomial $F$ in $n$ variables, let us define the set
$$
F(\bZ^n) = \{ F(x_1, \ldots, x_n):  (x_1, \ldots, x_n) \in \bZ^n   \}.
$$
Then we may reformulate Lagrange's theorem as follows.
\begin{theorem}[Lagrange, 1770]
Let $\mathcal{L}(w, x, y, z) = w^2 + x^2 + y^2 + z^2$. Then
$$
\mathcal{L}(\bZ^4) = \bZ_{\geq 0}.
$$
\end{theorem}
The problem of representing integers becomes more challenging with fewer variables. Considered one of the classical results in number theory, Gau{\ss} proved that every non-negative integer is a sum of three triangular numbers.
\begin{theorem}[Gau{\ss}, 1796]
Let $\mathcal{G}(x, y, z) = \frac{x(x+1)}{2} + \frac{y(y+1)}{2} + \frac{ z (z+1) }{2}$. Then
$$
\mathcal{G}(\bZ^3) = \bZ_{\geq 0}.
$$
\end{theorem}
These results by Lagrange and Gau{\ss} are special cases of what is known as Cauchy's polygonal number theorem, for which we
refer the reader to \cite{Na1} for a summary. More generally, these results are all within
an important area in number theory called additive bases, which includes problems such as
Goldbach's conjecture and Waring's problem. We refer to \cite{Na2} for a comprehensive introduction to this topic.
One consequence of these results is the following: Given $n > 2$, there exists a rational polynomial $F$ in $n$ variables
such that
$$
F(\bZ^n) = \bZ_{\geq 0}.
$$
It is clear that such $F$ does not exist when $n = 1$.
Therefore, it is natural to ask what happens when $n = 2$, which is precisely the question of John S. Lew and Bjorn Poonen.
\begin{question}[Lew-Poonen]
\label{QPoon}
Does there exist $F \in \bQ[x,y]$ such that
$$
F(\bZ^2) = \bZ_{\geq 0} \, ?
$$
\end{question}
This question was asked by Bjorn Poonen on the popular website MathOverflow\footnote{\url{https://mathoverflow.net/questions/9731/polynomial-representing-all-nonnegative-integers}}
in 2009 and has generated a great deal of interest from many mathematicians active on the website. Despite this, there has been no progress (at least that we are aware of).
In fact, the problem goes back at least to 1981. In an investigation of computer storage of multi-dimensional arrays,
finding a  polynomial in two variables that is ``approximately'' a bijection from $(\mathbb{Z}_{\geq 0})^2 \to \mathbb{Z}_{\geq 0}$ was important,
and in this relation Question \ref{QPoon} was posed by Lew in \cite{L}. Inspired by this question, Richard Stanley also asked\footnote{\url{https://mathoverflow.net/questions/45511/density-of-values-of-polynomials-in-two-variables}} the following regarding
the growth rate of these polynomials. Given any set $\mathfrak{B} \subseteq \bR$, let us define
$$
\mathcal{R}_F (\mathfrak{B}) =   F(\bZ^2) \cap \mathfrak{B}. 
$$
\begin{question}[Stanley]
\label{QRich}
Suppose $F \in \bQ[x,y]$ satisfies
$$
F(\bZ^2) \subseteq \bZ_{\geq 0}.
$$
Does it then follow that
$$
\# \mathcal{R}_F ([1, N]) \ll  \frac{N}{\sqrt{\log N}} \, ?
$$
\end{question}
We remark that the bound given is best possible as the well-known result by Landau \cite{La} establishes that
if $F(x,y) = x^2 + y^2$, then there exists $C(F) > 0$ such that
$$
\# \mathcal{R}_F ([1, N])  \sim C(F) \frac{N}{\sqrt{\log N}}.
$$

It is an easy exercise to see that there does not exist $F$ satisfying the condition in Question \ref{QPoon}
if $\deg F$ is odd. The same conclusion holds when $\deg F = 2$ as well (see Corollary \ref{quadpo}).
Therefore, the first real challenge of Question \ref{QPoon} is when $\deg F = 4$, which is what we answer in this paper.

\begin{theorem}
\label{mainthm}
There does not exist $F \in \bQ[x,y]$ of degree $4$ such that
$$
F(\bZ^2) = \bZ_{\geq 0}.
$$
\end{theorem}

As a consequence of our method, we also answer Question \ref{QRich} for irreducible quartic polynomials.
\begin{theorem}
Suppose $F \in \bQ[x,y]$ of $\deg F = 4$ is irreducible over $\bQ$ and satisfies
$$
F(\bZ^2) \subseteq \bZ_{\geq 0}.
$$
Then
$$
\# \mathcal{R}_F ([1, N]) \ll  \frac{N}{\sqrt{\log N}}.
$$
\end{theorem}

We remark that the counting function $\# \R_F([-N,N])$ in the case when $F \in \bZ[x,y]$ is a form has been the subject of an intensive study.
We refer the reader to \cite{SX} for a summary,
in which the first author and Stewart established the asymptotic formula for $\# \R_F([-N,N])$ when $\deg F \geq 3$ and the discriminant of $F$ is non-zero.


\subsection{Outline of the proof}
\label{outline}
For the remainder of the paper, we always assume $F \in \bZ[x,y]$ is of degree $4$ with $F(0,0) = 0$, and denote
$$
F(x,y) = \sum_{i=1}^4 F_i(x,y),
$$
where $F_i$ is the homogeneous degree $i$ part of $F$ (possibly the zero polynomial).
In order to prove Theorem \ref{mainthm}, it is sufficient to show that given any $F$,  
there do not exist $C \in \bN$ and $D \in \bZ$ such that
\begin{equation}
\label{maineq1}
F(\bZ^2) =  (C \bZ)_{\geq D} = \{ C n:  n \geq D/C \}.
\end{equation}
Our main strategy is to show that (\ref{maineq1}) can not occur due to one of the following three reasons:
\begin{enumerate}
  \item Large negative value (we have that $\inf_{(x,y) \in \bZ^2} F(x,y) = - \infty$).

  \item Density (we have that $\# \mathcal{R}_F  ([N, 2N]) = o(N)$).

  \item Reducibility of $F$ (the polynomial $F$ factors and  $\mathcal{R}_F(\mathbb{N})$ does not contain the set of prime numbers).
\end{enumerate}
It is clear that if any of these three conditions occur, then (\ref{maineq1}) does not hold.
For a ``random'' $F$, showing (\ref{maineq1}) is a relatively simple task. We proceed by eliminating these simple cases and arrive at $F$ with ``structure''.

If $F_4$ is positive semi-definite, then we let $\widetilde{F}_4 = F_4/P_0$, where $P_0$ is the positive definite integral form of the largest degree such that $P_0|F_4$
and $\widetilde{F}_4$ is a primitive integral form.
For simplicity, let us say $F$ is \textit{linked} if $F_4$ is positive semi-definite, $\gcd(\widetilde{F}_4, F_3) \neq 1$ and $\gcd(\widetilde{F}_4, F_3, F_2) = 1$.

We prove three main propositions in this paper, from which Theorem \ref{mainthm} follows.

\begin{proposition}
\label{ezprop1}
Let $F \in \bZ[x,y]$ be of degree $4$ with $F(0,0) = 0$ which is not linked.
Then
$$F(\bZ^2) \neq  (C\bZ)_{\geq D}$$ for any $C \in \bN$ and $D \in \bZ$.
\end{proposition}

\begin{proposition}
\label{ezprop2}
Let $F \in \bZ[x,y]$ be of degree $4$ with $F(0,0) = 0$ which is linked. Given any irreducible primitive integral form $H$ which divides $\gcd(\widetilde{F}_4, F_3)$,
let us denote $n_4(H)$ and $n_3(H)$ to be the largest integers such that  $H^{n_4(H)} | \widetilde{F}_4$  and $H^{n_3(H)} | F_3$ respectively.
If $n_3(H) \neq n_4(H)/2$ for all $H$, then $$F(\bZ^2) \neq  (C\bZ)_{\geq D}$$  for any $C \in \bN$ and $D \in \bZ$.
\end{proposition}
After establishing Propositions \ref{ezprop1} and \ref{ezprop2} in Sections \ref{sec2} and \ref{sec3} respectively, the remaining $F$ are of the form
\begin{eqnarray}
\label{diffF2}
F(x, y) = H (x, y)^{2r} G_{4}(x, y) + H (x, y)^{r} G_{3}(x, y) + F_2 (x,y) + F_1(x,y),
\end{eqnarray}
where $r \in \{1,2\}$, $F_4 = H^{2r} G_{4}$ is positive semi-definite, $H$ is either a primitive integral linear form or an indefinite irreducible primitive integral quadratic form, $G_4$ and $G_3$ are integral forms and $H \nmid G_4, G_3, F_2$. Let us call these polynomials given by (\ref{diffF2})
\textit{quadratically composed}. We then establish the following proposition by combining the main results of Sections \ref{sec4} and \ref{sec5}, namely Propositions \ref{first}, \ref{second} and \ref{lastprop}.

\begin{proposition}
\label{orderlyprop} Let $F \in \bZ[x,y]$ be of degree $4$ with $F(0,0) = 0$.
If $F$ is quadratically composed, then $$F(\bZ^2) \neq  (C\bZ)_{\geq D}$$  for any $C \in \bN$ and $D \in \bZ$.
\end{proposition}

As we shall see, quadratically composed $F$ are the most challenging to deal with, and the techniques from
Sections \ref{sec2} and \ref{sec3} struggle to yield the desired conclusions. As such, we need to take a different path for this final case,
a more algebraic approach which makes use of the additional structure; this is achieved in Sections \ref{sec4} and \ref{sec5}.
Finally, it is a natural question to ask whether our approach generalizes to higher degrees. This will be the content of our forthcoming paper.

\textit{Acknowledgements.} 
The second author was supported by the NWO Veni Grant \texttt{016.Veni.192.047} during his time at Utrecht University and by the FWF grant P 36278 at the Institute of Science and Technology Austria while working on this paper.
The first author would like to thank Samir Siksek for introducing the problem to him. 
The material in Section \ref{proofprime} is a result of discussing with many people, and the second author is very grateful to
Tim Browning, Stephanie Chan, Jakob Glas, Jakub L\"{o}wit, Mirko Mauri, Marta Pieropan, Mike Roth, Matteo Verzobio and Victor Wang for taking their time to answer his questions and their valuable suggestions. We thank Tim Browning, Yijie Diao and Ana Marija Vego for their comments on an earlier version of the paper.

\textit{Conventions and Notations.} We will use the convention that $N$ is always a sufficiently large positive integer and $B$ a sufficiently large positive real number, even if it is not explicitly stated so. We also use the convention that the degree of the zero polynomial is $0$.
Throughout we use $\ll$ and $\gg$ to denote Vinogradov's well-known notation.
By $f \asymp g$ we mean $f \ll g$ and $f \gg g$. We also make use of the $O$-notation and $o$-notation due to Landau.
Given $(x, y) \in \bR^2 \setminus \{ \mathbf{0} \}$, we denote $\arg(x, y)$ to be  the unique $\theta \in [0, 2 \pi)$ satisfying
$$
\frac{x}{\sqrt{x^2 + y^2}} = \cos (\theta) \quad \textnormal{and} \quad  \frac{ y }{\sqrt{x^2 + y^2}}  =  \sin (\theta).
$$
A form is a homogeneous polynomial and a primitive integral form is a form whose greatest common divisor of the coefficients is $1$.
Recall that given an integral form $G$ and $\bF = \bR$ or $\bZ$, we say $G$ is positive semi-definite, negative semi-definite or indefinite over $\bF$
if $G(x,y) \geq 0$ for all $(x,y) \in \bF^2$, $G(x,y) \leq 0$ for all $(x,y) \in \bF^2$ or we have $F (x_1, y_1) > 0$ and $F (x_2, y_2) < 0$ for some
$(x_1,y_1)$ and $(x_2, y_2) \in \bF^2 \setminus \{ \mathbf{0} \}$ respectively. The definition of positive (or negative) definiteness is the same as that of
positive (or negative) semi-definiteness except it is with a strict inequality. We note that\footnote{This follows easily from the fact that if there exists $(x_1, y_1) \in \bR^2$ such that the form is positive (or negative), then there also exists $(x_2, y_2) \in \bZ^2$ with the same property.}
given an integral form whether it is positive semi-definite, negative semi-definite or indefinite over $\bR$ is equivalent to that over $\bZ$.
By irreducible we mean irreducible over $\bQ$, unless stated otherwise.

\section{Proof of Proposition \ref{ezprop1}}
\label{sec2}	

Our goal in this section is to establish some basic results and prove Proposition \ref{ezprop1}.
We begin with the following simple observation, which we leave as a basic exercise to the reader.
\begin{lemma} \label{single}
Suppose there exists a polynomial $\mathfrak{F} \in \bQ[t]$ of degree $d \geq 2$ and $G \in \bQ[x,y]$ such that $F = \mathfrak{F} \circ G$. Then
\[\# \R_F([N, 2N]) \ll N^{\frac{1}{d}}.\]
\end{lemma}

\begin{lemma} \label{posdef lem}
Suppose $F_4$ is  positive definite. Then
$$
\# \mathcal{R}_F([N, 2N]) \ll \sqrt{N}.
$$
\end{lemma}
	
	\begin{proof}
Since $F_4$ is positive definite, there exists a  constant $c > 0$ such that
	\[F_4(x,y) \geq c \max\{|x|, |y|\}^4  \]
	for all $(x,y) \in \bZ^2$.
It is clear that
$$
F(x,y) =  F_4 (x,y) + O( \max\{|x|, |y|\}^3  ).
$$
Let $N$ be a sufficiently large positive integer.
Then we obtain
\begin{eqnarray}
\# \mathcal{R}_F([N, 2N])
\notag
&=&
\# \{(x,y) \in \bZ^2 : N \leq F(x,y) \leq 2 N\}
\\
\notag
&\leq& \# \left\{(x,y) \in \bZ^2 : \max\{|x|, |y|\}  \ll N^{\frac{1}{4}}\right\}
\\
\notag
&\ll& \sqrt{N}.
\end{eqnarray}
\end{proof}

	
\begin{lemma} \label{indef lem} Suppose $F_{4}$ is negative semi-definite or indefinite. Then
$$
\inf_{(x,y) \in \bZ^2} F(x, y) = - \infty.
$$
\end{lemma}
\begin{proof}
Since $F_{4}$ is negative definite or indefinite, there exists $(x_0, y_0) \in \bZ^2$ such that
$$
F_{4}(x_0, y_0) < 0.
$$
It is clear that
$$
F(N x_0, N y_0) = N^4 F_4(x_0, y_0) + O(N^3),
$$
and the result follows.
\end{proof}	

We conclude from Lemmas \ref{posdef lem} and \ref{indef lem} that $F_{4}$ must be positive semi-definite.
For the remainder of the paper, we will assume that this is the case.
We leave the following fact as an exercise to the reader. Furthermore, we define $\widetilde{F}_4$ to be as in the statement\footnote{This definition is equivalent to that given in Section \ref{outline}. }.
\begin{lemma}
\label{possemi}
If $F_{4}$ is positive semi-definite, then there exists a positive definite integral form $P_0$ (possibly a positive constant) such that
$\widetilde{F}_4 = F_4/P_0$ is a non-empty product of irreducible primitive integral forms $H$ with a non-trivial real zero and the multiplicity of each $H$ is even.
\end{lemma}
	
Next we deal with the case $\gcd( \widetilde{F}_{4}, F_{3}) = 1$.  
\begin{lemma} \label{gcd}
Suppose $F_{4}$ is positive semi-definite  and $\gcd(\widetilde{F}_{4}, F_{3}) = 1$.
Then
$$
\inf_{(x,y) \in \bZ^2} F(x, y) = - \infty.
$$
\end{lemma}
\begin{proof}
We begin by noting that the hypotheses of the statement ensure that $F_3$ is not the zero polynomial.
Let $(\cos (\xi), \sin(\xi))$ be a zero of $\widetilde{F}_{4}$, which we know to exist by Lemma \ref{possemi},
and write
$$
\mathfrak{L}(x,y) = \sin (\xi) x - \cos (\xi) y
$$
and
$$
F_4(x, y) = \mathfrak{L}(x, y)^{2 r} \mathfrak{F}_{4} (x, y),
$$
where $r \in \mathbb{N}$ and $\mathfrak{L} \nmid \mathfrak{F}_{4}$. In particular, since $\gcd(\widetilde{F}_{4}, F_{3}) = 1$, we have $\mathfrak{L} \nmid F_{3}$.
By Dirichlet's theorem on Diophantine approximation,
there exist infinitely many $(u, v) \in \bZ^2$ satisfying
$$
|\sin (\xi) u - \cos (\xi) v| \ll \frac{1}{\sqrt{ u^2 + v^2 }}.
$$
For these $(u, v)$, we have
$$
F_{4}(u, v) = \mathfrak{L}(u, v)^{2 r} \mathfrak{F}_{4} (u, v)  \ll ( u^2 + v^2 )^{- r} (u^2 + v^2 )^{2 -  r} \ll (u^2 + v^2)^{2 - 2 r} \ll 1,
$$
$$
| F_{3}(u, v) |  \gg  (u^2 + v^2)^{ \frac{3}{2} }
\quad
\textnormal{and}
\quad
|F_{2}(u, v) + F_1(u, v)| \ll u^2 + v^2.
$$
Therefore, we obtain
$$
F (u, v) = F_{3}(u, v) + O(u^2 + v^2).
$$
By replacing $(u,v)$ with $(-u, - v)$ if necessary, it follows that
$$
F (u, v) < 0
\quad
\textnormal{and}
\quad
|F(u, v)| \gg (u^2 + v^2)^{ \frac{3}{2}},
$$
provided $\sqrt{u^2 + v^2}$ is sufficiently large, and the result follows.
\end{proof}	

By Lemma \ref{gcd} we may assume that $\gcd( \widetilde{F}_4, F_3) \ne 1$ from this point on.
We now deal with the case $\gcd( \widetilde{F}_4, F_3, F_2) \ne 1$ but $\gcd( \widetilde{F}_4, F_3, F_2, F_1) =  1$.
We are then left with the cases $\gcd(\widetilde{F}_4, F_3, F_2) = 1$, that is $F$ is linked,
and $\gcd(\widetilde{F}_4, F_3, F_2, F_1) \neq 1$.
	
\begin{lemma}
\label{baah}
Suppose $F_4$ is positive semi-definite,  $\gcd( \widetilde{F}_4, F_3, F_2) \ne 1$  and  $\gcd(\widetilde{F}_4, F_3, F_2, F_1) =  1$.
Then
$$
\inf_{(x,y) \in \bZ^2} F(x, y) = - \infty.
$$
\end{lemma}
\begin{proof}
We begin by noting that the hypotheses of the statement ensure that $F_1$ is not the zero polynomial.
Let $(\cos (\xi), \sin(\xi))$ be a zero of $\gcd( \widetilde{F}_4, F_3, F_2)$.
We write
$$
\mathfrak{L}(x,y) = \sin (\xi) x - \cos (\xi) y,
$$
$$
F_{4}(x, y) = \mathfrak{L}(x, y)^{2 r} \mathfrak{F}_{4} (x, y)
\quad
\textnormal{and}
\quad
F_{2}(x, y) = \mathfrak{L}(x, y)^{t} \mathfrak{F}_{2} (x, y),
$$
where $r, t \in \mathbb{N}$ and $\mathfrak{L} \nmid \mathfrak{F}_{4}, \mathfrak{F}_2$ (If $F_2$ happens to be the zero polynomial, then
we set $t = 0$ and $\mathfrak{F}_2$ to be the zero polynomial.).

By Dirichlet's theorem on Diophantine approximation, there exist infinitely many $(u, v) \in \bZ^2$ satisfying
$$
| \sin (\xi) u  -  \cos (\xi) v | \ll  \frac{1}{\sqrt{ u^2 + v^2 } }.
$$
For these $(u, v)$, we have
$$
F_{4}(u, v) = \mathfrak{L}(u, v)^{2 r} \mathfrak{F}_{4} (u, v)  \ll ( u^2 + v^2 )^{- r} (u^2 + v^2)^{2 -  r} \ll (u^2 + v^2)^{2 - 2r} \ll 1,
$$
$$
| F_{2}(u, v) | =  |\mathfrak{L}(u, v)|^{t} |\mathfrak{F}_{2} (u, v)| \ll (u^2 + v^2)^{- \frac{t}{2}} (u^2 + v^2)^{1 - \frac{t}{2}} \ll (u^2 + v^2)^{1 - t} \ll 1
$$
and
$$
\sqrt{u^2 + v^2} \ll | F_{1}(u, v) | \ll \sqrt{u^2 + v^2}.
$$
Therefore, we obtain
$$
F (u, v) = F_3(u,v) + F_{1}(u, v) + O(1).
$$
If $\{ F_3(u, v) + F_1(u, v) : (u,v) \}$ is an unbounded set, then by replacing $(u,v)$ with $(-u, -v)$ if necessary,
our result follows.  Suppose otherwise, i.e. $|F_3(u, v) + F_1(u, v) | \ll 1$ for all $(u, v)$ in consideration, which implies
$$
\sqrt{u^2 + v^2} \ll F_3(u, v) \ll \sqrt{u^2 + v^2}.
$$
Given any $M > 1$, we have
$$
F(M u, M v) = M^3 F_3(u, v) + M F_1(u, v) + O(M^4).
$$
Therefore, by setting $M = \lfloor (u^2 + v^2)^{\varepsilon_0/2} \rfloor$ with $\varepsilon_0 > 0$ sufficiently small, and again replacing $(u,v)$ with $(-u, -v)$ if necessary, we obtain our result in this case as well.
\end{proof}

Recall that the remaining cases are $F$ is linked and $\gcd(\widetilde{F}_4, F_3, F_2, F_1) \neq 1$.
We will prove the following in Section \ref{proofprime}, which completes our proof of Proposition \ref{ezprop1}.
\begin{proposition}
\label{prime}
Suppose $F_4$ is positive semi-definite and  $\gcd( \widetilde{F}_4, F_3, F_2, F_1) \ne 1$.
Let $C \in \bN$ and $D \in \bZ$.
Then
$$
F(\bZ^2) \neq  (C\bZ)_{\geq D}.
$$
\end{proposition}

In this case, $F$ is reducible over $\bQ$ and the proof of Proposition \ref{prime} requires several results from Diophantine geometry.
We note that the material in Sections \ref{sec3}-\ref{sec5}  are  independent of that in Section \ref{proofprime}.


\section{Proof of Proposition \ref{ezprop2}}
\label{sec3}
Our goal in this section is to prove Proposition \ref{ezprop2}. We begin by setting some notations for the remainder of the paper.
Let $F_4$ be positive semi-definite and $\gcd(\widetilde{F}_4, F_3) \neq 1$. Let $\bT = \bR/ (2 \pi \bZ).$
We define $\Xi$ to be the set of all $\xi \in \bT$ such that
$$
\widetilde{F}_{4}( \cos (\xi), \sin (\xi) ) = 0
\quad
\textnormal{and}
\quad
F_{3}( \cos (\xi), \sin (\xi) ) = 0.
$$
Let $\xi \in \Xi$ and $H$ be the irreducible primitive integral form $H$ satisfying
$H(\cos(\xi), \sin (\xi)) = 0$.  By Lemma \ref{possemi}, it follows that $H$ divides $F_4$ with an even order, which we denote by $2 r_\xi$;
when we are working with a fixed choice of $\xi$, we will drop the subscript and simply denote $r$ for a clearer exposition.
We define
$$
\Theta_{\xi}(B) =  \left\{  \theta \in \bT:  |\theta - \xi| < c_\xi B^{ - \frac{1}{2r_\xi} }  \right\},
$$
where $c_{\xi} > 0$ is a sufficiently large constant. Then we have the following simple estimate, which we leave as an exercise to the reader.

\begin{lemma}
Let $F$ be as in this section.
Then there exists $C_1 > 0$ such that we have
\begin{eqnarray}
\label{BDD}
F (B \cos (\theta), B \sin(\theta)) \geq C_1 B^{3},
\end{eqnarray}
for all
$$
\theta \in \bT \setminus  \bigcup_{\xi \in \Xi}  \Theta_{\xi}(B).
$$
\end{lemma}

The idea is that once we are sufficiently ``far'' from the zeros of $F_4$, the situation is easy to control, and the challenge comes from ``near'' the zeros.
We now observe the following fact, which we will make use of throughout the remainder of the paper.

\begin{proposition}
\label{2222}
Let $F$ be as in this section. Let $\mathcal{E} \subseteq \bZ^2$.
Suppose for each $\xi \in \Xi$, there exists $C_{\xi} >0$ such that
\begin{eqnarray}
\label{LOWER}
F_4(B \cos (\theta), B \sin(\theta))
+
F_3 (B \cos (\theta), B \sin(\theta))
+ F_2 (B \cos (\theta), B \sin(\theta))
> C_{\xi} B^{2},
\end{eqnarray}
for all
$$
(B \cos (\theta), B \sin(\theta)) \in \bZ^2 \setminus \mathcal{E}
$$
with $\theta \in \Theta_\xi(B)$ and $B \gg 1$.
Then there exists $\lambda > 0$ such that
$$
\# \mathcal{R}_F([N, 2N]) \ll N^{1 - \lambda} + \# (\{ F(x,y):  (x,y) \in \mathcal{E} \} \cap [N, 2N]).
$$
\end{proposition}
We note that $C_\xi$ depends on $c_\xi$ from the definition of $\Theta_\xi(B)$, but it is independent of $B$.
Also, we shall take $\mathcal{E}$ in the statement to be the empty set for all applications of this result in this section.

\begin{proof}
Let $(x, y) \in \bZ^2$ and suppose $N \leq F(x,y) \leq 2 N$.
By (\ref{BDD}) it follows that
$$
C_1 (x^2 + y^2)^{ \frac{3}{2} } \leq F(x,y) \leq 2 N
$$
if
$$
\arg (x, y) \in \bT \setminus \bigcup_{\xi \in \Xi}  \Theta_{\xi}(\sqrt{x^2 + y^2}).
$$
On the other hand, by  (\ref{LOWER}) it follows that there exists $C_2 > 0$ satisfying
$$
C_2  (x^2 + y^2)  \leq F(x,y) \leq 2N
$$
if $(x, y) \not \in \mathcal{E}$ and
$$
\arg (x, y) \in \bigcup_{\xi \in \Xi}  \Theta_{\xi}(\sqrt{x^2 + y^2}).
$$
Let
$$
R = \max_{\xi \in \Xi} r_\xi.
$$
Therefore, we obtain
\begin{eqnarray}
&& \# \mathcal{R}_F([N, 2N]) -  \# (\{ F(x,y):  (x,y) \in \mathcal{E} \} \cap [N, 2N])
\notag
\\
&\leq& \# \{ (x, y) \in \bZ^2 \setminus \mathcal{E}:  N \leq F(x,y) \leq 2 N \}
\notag
\\
\notag
&\ll&
\# \{ (x, y) \in \bZ^2:   \sqrt{x^2 + y^2}  \leq (2 N/C_1)^{ \frac{1}{3} }  \}
\\
\notag
&& +
\sum_{\xi \in \Xi}
\# \{ (x, y) \in \bZ^2: \sqrt{x^2 + y^2}  \leq (2 N/C_2)^{ \frac{1}{2} },
|\arg (x, y) - \xi| < c_{\xi} (x^2 + y^2)^{- \frac{1}{4 R}}   \}
\\
\notag
&\ll& N^{\frac{2}{3} } + N^{1 - \frac{1}{4 R}},
\end{eqnarray}
where the second term in the final inequality may be obtained as follows. Let
$$
Z(T) = \# \{ (x, y) \in \bZ^2:  T \leq   \sqrt{x^2 + y^2}  \leq 2 T,  |\arg (x, y) - \xi| \ll T^{- \frac{1}{2 R}}   \}.
$$
This quantity is bounded by the sum of the area of the region and the length of the boundary, to be more precise
$$
Z(T) \ll  T^{2 - \frac{1}{2R} }   +  T.
$$
Then by considering dyadic intervals, we obtain
\begin{eqnarray}
&&\# \{ (x, y) \in \bZ^2:   \sqrt{x^2 + y^2}  \leq (2 N/C_2)^{ \frac{1}{2} },
|\arg (x, y) - \xi| < c_{\xi} (x^2 + y^2)^{- \frac{1}{4 R}}   \}
\notag
\\
\notag
&\ll&
\sum_{1 \leq j \ll \log N } Z((2 N/C_2)^{ \frac{1}{2} } 2^{-j})
\\
\notag
&\ll&
N^{1 - \frac{1}{4 R}} + N^{\frac{1}{2}}.
\end{eqnarray}
\end{proof}

\begin{remark}
\label{REMAA}
The way we use this proposition is as follows. We will prove that for each $\xi \in \Xi$, either
(\ref{LOWER}) holds or we have
$$
\inf_{(x, y) \in \bZ^2 } F(x, y) = - \infty.
$$
If the latter happens for at least one $\xi$, then we are done. Otherwise, it means (\ref{LOWER}) holds for all $\xi \in \Xi$,
in which case by Proposition \ref{2222} it follows that there exists $\lambda > 0$ such that
$$
\# \mathcal{R}_F([N, 2N]) \ll N^{1 - \lambda} + \# (\{ F(x,y):  (x,y) \in \mathcal{E} \} \cap [N, 2N]).
$$
\end{remark}


From this point on, we assume that $F$ is linked.
In particular, it follows that $F_2$ is not the zero polynomial while $F_3$ may be. We begin with a simple consequence of Proposition \ref{2222}, which deals with the case when $F_3$ is the zero polynomial.
\begin{lemma}
\label{lemm}
Suppose $F_2(\cos(\xi), \sin(\xi)) < 0$ for some $\xi \in \Xi$. Then
$$
\inf_{(x, y) \in \bZ^2 } F(x, y) = - \infty.
$$
Furthermore, if $F_2(\cos(\xi), \sin(\xi)) > 0$ for all $\xi \in \Xi$ and
$F_3$ is the zero polynomial, then there exists $\lambda > 0$ such that
$$
\# \mathcal{R}_F([N, 2N]) \ll N^{1 - \lambda}.
$$
\end{lemma}
\begin{proof}
Suppose $F_2(\cos(\xi), \sin(\xi)) < 0$ for some $\xi \in \Xi$.
By Dirichlet's theorem on Diophantine approximation, there exist infinitely many $(u, v) \in \bZ^2$ satisfying
\begin{eqnarray}
\label{DD1}
| \sin (\xi) u  -  \cos (\xi) v | \ll \frac{1}{ \sqrt{ u^2 + v^2  } }.
\end{eqnarray}
For these $(u, v)$, a similar argument as in the proof of Lemma \ref{baah} yields
$$
F_{4}(u, v)  \ll 1,
\quad
|F_3(u, v)| \ll  \sqrt{ u^2 + v^2  }
$$
and
$$
|F_2 (u, v)| \gg  u^2 + v^2.
$$
Furthermore, by the mean value theorem and (\ref{DD1}) we have
$$
| F_2(\cos(\xi), \sin (\xi)) - F_2 (\cos(\arg(u,v)), \sin (\arg(u,v)) )| \ll  |\xi - \arg(u,v)| \ll \frac{1}{ u^2 + v^2   }.
$$
In particular, it follows that  $F_2(u,v) < 0$, provided $\sqrt{u^2 + v^2}$ is sufficiently large.
Therefore, we obtain
$$
F(u,v) =  F_{2}(u, v) + O( \sqrt{u^2 + v^2} ),
$$
from which it easily follows that
$$
\inf_{(x,y) \in \bZ^2} F(x, y) = - \infty.
$$


Next we suppose $F_2(\cos(\xi), \sin(\xi)) > 0$ for all $\xi \in \Xi$ and  $F_3$ is the zero polynomial.
In particular, we have
\begin{eqnarray}
\label{+++}
F_2( B \cos(\theta), B  \sin (\theta)) = B^2 F_2(\cos(\theta), \sin (\theta)) > 0
\end{eqnarray}
for all $\theta \in \Theta_\xi(B)$.
Then it is clear that there exists $C_3 > 0$ satisfying
\begin{eqnarray}
\notag
&&F_4(B \cos (\theta), B \sin(\theta))
+
F_3 (B \cos (\theta), B \sin(\theta))
+ F_2 (B \cos (\theta), B \sin(\theta))
\\
\notag
&\geq&
F_2 (B \cos (\theta), B \sin(\theta))
\\
\notag
&>& C_3 B^{2}
\end{eqnarray}
for all $\theta \in \Theta_\xi(B)$ and $\xi \in \Xi$.
Therefore, we obtain that (\ref{LOWER}) holds for all $\xi \in \Xi$ and the conclusion follows by Proposition \ref{2222}.
\end{proof}

For the remainder of this section, we assume that $F_2(\cos(\xi), \sin (\xi)) > 0$ for all $\xi \in \Xi$ and $F_3$ is not the zero polynomial.
In particular, we have (\ref{+++}) for all $\theta \in \Theta_\xi(B)$. Let us fix a choice of $\xi \in \Xi$.
We now set more notations for the remainder of the paper. We denote
$$
\mathfrak{L}(x, y) = \sin (\xi) x - \cos (\xi) y,
$$
$$
F_{4}(x, y) = \mathfrak{L} (x, y)^{2r} \mathfrak{F}_{4}(x, y)
\quad
\textnormal{and}
\quad
F_{3}(x, y) = \mathfrak{L} (x, y)^{s} \mathfrak{F}_{3}(x, y),
$$
where $r = r_\xi, s = s_\xi \in \bN$ and $\mathfrak{L} \nmid \mathfrak{F}_{4}, \mathfrak{F}_{3}, F_2$. 
In particular, we have
$$
\mathfrak{F}_{4}(\cos (\xi), \sin(\xi)), \,   \mathfrak{F}_{3}(\cos (\xi), \sin(\xi)) \neq 0.
$$
We also remark that $\mathfrak{F}_{4}$ and $\mathfrak{F}_3$ depend on the choice of $\xi$.
For simplicity, we define
$$
\delta = \delta(\theta, B) =  \frac{\log  |\mathfrak{L} (\cos (\theta), \sin (\theta))| }{ \log B }
$$
so that
$$
B^{\delta(\theta, B)} =  |\mathfrak{L} (\cos (\theta), \sin (\theta))|.
$$
Recalling the definition $\Theta_{\xi}(B) = \{ \theta \in \bT:  |\theta - \xi| < c_\xi B^{ - \frac{1}{2r} }  \}$, it is clear that
$$
|\mathfrak{L} (\cos (\theta), \sin (\theta))|
= |\mathfrak{L} (\cos (\xi), \sin (\xi)) - \mathfrak{L} (\cos (\theta), \sin (\theta))|
\leq |\cos (\xi) - \cos (\theta) | + |\sin (\xi) - \sin (\theta) |
\ll B^{ - \frac{1}{2r} }
$$
for all $\theta \in \Theta_{\xi}(B)$. Consequently, $\delta(\cdot, B)$ is a continuous function in $\theta \in \Theta_{\xi}(B) \setminus \{\xi\}$
with range
\begin{eqnarray}
\label{cont}
I_{\xi}  =  \left(- \infty, - \frac{1}{2 r} + o_B(1) \right].
\end{eqnarray}
With these notations, we have
$$
F_{4}(B \cos (\theta), B \sin (\theta)) \asymp B^{4 + 2r \delta}, \quad F_{3}(B \cos (\theta), B \sin (\theta)) \asymp B^{3 + s \delta}
$$
and
$$
F_{2}(B \cos (\theta), B \sin (\theta)) \asymp B^{2}
$$
for all $\theta \in \Theta_{\xi}(B)$.

We begin by considering the two cases when either $|F_{4}(B \cos (\theta), B \sin (\theta))|$ or $|F_{2}(B \cos (\theta), B \sin (\theta))|$
dominates $|F_{3}(B \cos (\theta), B \sin (\theta))|$. Then we consider the case when $|F_{3}(B \cos (\theta), B \sin (\theta))|$ dominates the other two terms.
Here we take the general approach described in Remark \ref{REMAA}.

Let $\varepsilon > 0$ be sufficiently small and $B$ sufficiently large with respect to $\varepsilon$.
Suppose $\theta \in \Theta_\xi(B)$ satisfies
\begin{eqnarray}
\label{del2}
\delta <  - \frac{1 + \varepsilon}{s},
\end{eqnarray}
which is equivalent to
$$
2 > 3  + s \delta + \varepsilon.
$$
Then we have
\begin{eqnarray}
&&F_{4}(B \cos (\theta), B \sin(\theta)) + F_{3}(B \cos (\theta), B \sin(\theta)) +  F_{2}(B \cos (\theta), B \sin(\theta))
\label{ineqlo}
\\
\notag
&=&B^{4 + 2r \delta} \mathfrak{F}_{4}(\cos (\theta), \sin(\theta))
+
B^{3 + s \delta} \mathfrak{F}_{3}(\cos (\theta), \sin(\theta))
+
B^{2} F_{2}(\cos (\theta), \sin(\theta))
\\
\notag
&\geq&
B^{4 + 2r \delta} \mathfrak{F}_{4}(\cos (\theta), \sin(\theta)) +  \frac{B^{2}}{2} F_{2}(\cos (\theta), \sin(\theta))
\\
\notag
&\gg&
B^{4 + 2 r \delta} + B^{2}.
\end{eqnarray}
Next suppose $\theta \in \Theta_\xi(B)$ satisfies
\begin{eqnarray}
\label{del1}
(2r - s) \delta >  - 1 + \varepsilon,
\end{eqnarray}
which is equivalent to
$$
4 + 2r \delta > 3  + s \delta + \varepsilon.
$$
Then we have
\begin{eqnarray}
&&F_{4}(B \cos (\theta), B \sin(\theta)) + F_{3}(B \cos (\theta), B \sin(\theta)) +  F_{2}(B \cos (\theta), B \sin(\theta))
\\
\notag
&=& B^{4  + 2r \delta} \mathfrak{F}_{4}(\cos (\theta), \sin(\theta))
+
B^{3 + s \delta} \mathfrak{F}_{3}(\cos (\theta), \sin(\theta))
+
B^{2} F_{2}(\cos (\theta), \sin(\theta))
\\
\notag
&\geq&
\frac{ B^{4 + 2 r \delta} }{ 2 }  \mathfrak{F}_{4}(\cos (\theta), \sin(\theta)) +  B^{2} F_{2}( \cos (\theta), \sin(\theta))
\\
\notag
&\gg&
B^{4 + 2 r \delta} + B^{2}.
\end{eqnarray}

With these estimates, we can collect the following.
\begin{lemma}
For $\xi \in \Xi$ with $r_\xi < s_\xi$, we have that (\ref{LOWER}) holds.
\end{lemma}
\begin{proof}
Clearly, the inequality
$$
 - \frac{1}{2r} + o_B(1) < - \frac{1 + \varepsilon}{ s }
$$
holds for $r, s \in \bN$ if and only if $2 r < s$. Therefore, if $2 r < s$, then
$$
I_\xi \subseteq  \left[ - \infty,   - \frac{1 + \varepsilon}{s}\right],
$$
that is every $\theta \in \Theta_\xi (B)$ satisfies (\ref{del2}). Thus for $\xi \in \Xi$ with $2 r_\xi < s_\xi$, we have that (\ref{LOWER}) holds.

Next it is easy to see that every $\theta \in \Theta_\xi(B)$ satisfies (\ref{del1}) if $2r = s$; therefore,
for $\xi \in \Xi$ with $2 r_\xi  = s_\xi$, we have that (\ref{LOWER}) holds.

We are left to deal with the case $2r > s$. Since the inequality
$$
- \frac{1 + \varepsilon}{ s } \leq   - \frac{1 - \varepsilon}{ 2r - s }
$$
is equivalent to
$$
r - s \geq - \varepsilon r,
$$
it is clear that $r, s \in \bN$, with $2r > s$, satisfy the former if and only if $r \geq s$. This means if
$r < s$, then given any $z \in \bR$ we have that either
$$
z < -  \frac{1+ \varepsilon}{s}
\quad
\textnormal{or}
\quad
z > - \frac{1 - \varepsilon}{2r-s}
$$
holds. In particular, every $\theta \in \Theta_\xi(B)$ satisfies either (\ref{del2}) or  (\ref{del1}).
Therefore, for $\xi \in \Xi$ with $r_\xi < s_\xi$ and $2 r_\xi > s_\xi$, we have that (\ref{LOWER}) holds.
\end{proof}

Thus the remaining case is $r \geq s$. Let us suppose $r > s$. We recall (\ref{cont}) and note that the following inequalities
$$
- \frac{1}{s} < - \frac{1}{s + \frac12 } < - \frac{1}{r} < - \frac{1}{2 r - s}
$$
hold in this case. Suppose $\theta_0 \in \Theta_{\xi}(B)$ satisfies
\begin{eqnarray}
\label{del eq}
\left| \delta_0 + \frac{1}{s + \frac12} \right| < \varepsilon,
\end{eqnarray}
where $\delta_0 = \delta(\theta_0, B)$.
In particular, $\delta_0$ satisfies
$$
- \frac{1 }{ s }  <  \delta_0 <   - \frac{1 }{ 2r - s },
$$
which implies
$$
4 + 2r \delta_0 + \varepsilon_0 < 3  + s \delta_0
$$
and
$$
2  +  \varepsilon_0 < 3  + s \delta_0,
$$
for $\varepsilon_0 > 0$ sufficiently small.
In this case, we have
\begin{eqnarray}
\label{eqn 1111}
&& F_{4}(B\cos (\theta_0), B \sin(\theta_0)) + F_{3}(B \cos (\theta_0), B \sin(\theta_0)) + F_{2}(B \cos (\theta_0), B \sin(\theta_0))
\\
\notag
&=&
B^{4 + 2r \delta_0} \mathfrak{F}_{4}(\cos (\theta_0), \sin(\theta_0))
+
B^{3 + s \delta_0} \mathfrak{F}_{3}(\cos (\theta_0), \sin(\theta_0))
+
B^{2} F_{2}(\cos (\theta_0), \sin(\theta_0))
\\
\notag
&=&
B^{3 + s \delta_0} \mathfrak{F}_{3}(\cos (\theta_0), \sin(\theta_0))
+
O( B^{4 + 2r \delta_0} + B^{2})
\\
\notag
&=&
F_{3}(B \cos (\theta_0), B \sin(\theta_0)) + O( B^{4 + 2r \delta_0} + B^{2}).
\end{eqnarray}

\begin{lemma}
\label{cont2}
Suppose $r_{\xi} > s_\xi$ for some $\xi \in \Xi$. Then
$$
\inf_{(x,y) \in \bZ^2} F(x,y) = - \infty.
$$
\end{lemma}
\begin{proof}
We shall consider two cases depending on whether $\cos (\xi)/ \sin (\xi) \in \mathbb{Q} \cup \{ \infty\}$ or not.
Let us begin with the case $\cos (\xi)/\sin (\xi) \notin \mathbb{Q} \cup \{ \infty \}$.
Let $\varepsilon > 0$ be sufficiently small. By Roth's theorem and Dirichlet's theorem on Diophantine approximation,
there exist infinitely many $(u, v) \in \bZ^2$ 
satisfying
$$
\frac{1}{|v|^{2 + \varepsilon}} \ll  \left| \frac{\cos (\xi)}{\sin (\xi)} -  \frac{u}{v} \right| < \frac{1}{|v|^2},
$$
which implies
$$
\frac{1}{ |v|^{1 + \varepsilon}} \ll  | \mathfrak{L}(u, v) | < \frac{1}{|v|}.
$$
For simplicity, we denote
$$
\mathfrak{R}(u,v) = \sqrt{u^2 + v^2}.
$$
We now set
$$
M =  \lfloor   \mathfrak{R}(u,v)^{2 s} \rfloor
$$
so that
$$
\frac{M}{ \mathfrak{R}(u,v) }  \asymp   
\mathfrak{R}(M u, M v)^{1 - \frac{2 }{ 2 s + 1} }.
$$
Then we have
\begin{eqnarray}
\notag
\mathfrak{R}(M u, M v)^{1 - \frac{2}{2 s + 1} - \varepsilon}
\ll
\frac{M}{|v|^{1 + \varepsilon}} \ll
| \mathfrak{L}(M u, M v )| \ll \frac{M}{|v|} \ll  \mathfrak{R}(Mu, M v)^{1 - \frac{2}{2 s + 1} }.
\end{eqnarray}
In particular, $\theta_0 = \arg(Mu , Mv)$ with $B = \mathfrak{R}(Mu, Mv)$ satisfies (\ref{del eq}), provided $\mathfrak{R}(u,v)$ is sufficiently large, which implies
(\ref{eqn 1111}). By replacing $(u, v)$ with $(- u, - v)$ if necessary, it follows that
$$
\inf_{(x,y) \in \bZ^2} F(x,y) = - \infty.
$$

On the other hand, let us now suppose $\cos (\xi)/\sin (\xi) \in \mathbb{Q} \cup \{ \infty \}$.
Without loss of generality we assume $\mathfrak{L} (1, 0) \neq 0$.
Let $(x_0, y_0) \in \mathbb{Z}^2 \setminus \{ \mathbf{0} \}$ be a zero of $\mathfrak{L}$.
Let $M \in \mathbb{N}$ be sufficiently large and
$$
(x, y) = \left(  M x_0  +   \lfloor  M^{1 - \frac{2}{ 2 s +  1 } } \rfloor, M y_0 \right).
$$
Then we have
$$
\mathfrak{R}(x, y) = M \mathfrak{R}(x_0, y_0) (1 + o_M(1)  )
$$
and
$$
\mathfrak{R}(x, y)^{1 - \frac{2 }{ 2 s +  1 } }  \ll   M^{1 - \frac{2 }{ 2 s +  1 }} \ll \mathfrak{L}(x,y) \ll  M^{1 - \frac{2 }{ 2 s +  1 }} \ll \mathfrak{R}(x, y)^{1 - \frac{2}{ 2 s +  1 } }.
$$
In particular, $\theta_0 = \arg(x, y)$ with $B = \mathfrak{R}(x, y)$ satisfies (\ref{del eq}), provided $\mathfrak{R}(x, y)$ is sufficiently large, which implies (\ref{eqn 1111}).
Therefore, by replacing $(x, y)$ with $(- x, - y)$ if necessary, the result follows.
\end{proof}

Finally, we combine all the work in this section, and recalling Remark \ref{REMAA}, to deduce the following.
\begin{proposition}
\label{random}
Let $F$ be as in this section (with the additional assumptions described after the proof of Lemma \ref{lemm}.). Suppose $r_{\xi} \neq s_\xi$ for all $\xi \in \Xi$.
Then either there exists $\lambda > 0$ such that
$$
\# \mathcal{R}_F([N, 2N]) \ll N^{1 - \lambda},
$$
or we have
$$
\inf_{(x, y) \in \bZ^2 } F(x, y) = - \infty.
$$
\end{proposition}


It is clear that Proposition \ref{random} implies Proposition \ref{ezprop2}.
Therefore, the final case we are left with is the following: there exists  $\xi \in \Xi$ such that
\begin{eqnarray}
\notag
F(x, y) = \mathfrak{L} (x, y)^{2r} \mathfrak{F}_{4}(x, y) + \mathfrak{L} (x, y)^{r} \mathfrak{F}_{3}(x, y) + F_2 (x,y) + F_1(x,y),
\end{eqnarray}
from which it follows that we may rewrite $F$ as in (\ref{diffF2}).
We deal with these in the following two sections. In this case, it is possible that  $\delta = - 1/r$, which implies
$$
|F_{4}(B \cos (\theta), B \sin (\theta))|, |F_{3}(B \cos (\theta), B \sin (\theta))|, |F_{2}(B \cos (\theta), B \sin (\theta))|
\asymp
B^{2},
$$
that is the three terms are of the same magnitude. Since this allows for the possibility of cancellation among the three terms,
the techniques from this and the previous sections struggle to yield the desired conclusions.
However, by making a full use of the quadratically composed structure of $F$, we can obtain the result.

\section{Quadratically composed polynomials with $H$ linear}
\label{sec4}
It follows from (\ref{diffF2}) that there are three cases to consider:
\begin{enumerate}
  \item $H$ is a linear form and $r = 1$.
  \item $H$ is a linear form and $r = 2$.
  \item $H$ is an indefinite quadratic form and $r = 1$.
\end{enumerate}
We shall treat these three cases separately.
In this section, we prove Propositions \ref{first} and \ref{second} which deal with (1) and (2) respectively,
and Proposition \ref{lastprop} which deals with (3) in the next section.
We will prove these in a slightly more general setting than in (\ref{diffF2}).  
It is easy to see that Proposition \ref{orderlyprop} from Propositions \ref{first}, \ref{second} and \ref{lastprop}.

\begin{proposition}
\label{first}
Let
$$
F(x, y) = H (x, y)^{2} G_{4}(x, y) + H (x, y) G_{3}(x, y) + F_2 (x,y) + F_1(x,y),
$$
where $H$ is a primitive integral linear form, $G_4$ and $G_3$ are integral quadratic forms such that $H \nmid G_4, G_3$,
there does not exist a primitive integral linear form $L \nmid H$ such that $L^2|G_3$ and
no further assumptions on $F_2$ and $F_1$ (In particular, this contains $F$ as in (\ref{diffF2}) with $H$ a linear form and $r = 1$.).
Then either there exists $\lambda > 0$ such that
$$
\# \mathcal{R}_F([N, 2N]) \ll N^{1 - \lambda},
$$
or we have
$$
\inf_{(x, y) \in \bZ^2 } F(x, y) = - \infty.
$$
\end{proposition}

\begin{proof}
Let $H(x,y) = a x + b y$. Since $\gcd(a, b) = 1$, there exist $c, d \in \bZ$ such that $a d - b c = 1$. Let us define the matrix
$$
A = \begin{bmatrix}
a   &  b \\
c &  d
\end{bmatrix}.
$$
In particular, if
$$
\begin{bmatrix}
  u \\
  v
\end{bmatrix}
= A
\begin{bmatrix}
  x \\
  y
\end{bmatrix},
$$
then we have $u = H(x,y)$.
Let $G = F \circ A^{-1}$ and denote
\begin{equation}
\label{case2shape1}
G(u, v) = u^2 \mathcal{G}_4(u, v) + u  \mathcal{G}_3(u, v) + \mathcal{G}_2(u, v) + \mathcal{G}_1(u, v),
\end{equation}
where $\mathcal{G}_4, \mathcal{G}_3$ are integral quadratic forms not divisible by $u$, and $\mathcal{G}_2$ and $\mathcal{G}_1$ are integral forms (possibly the zero polynomial) of degrees $2$ and $1$ respectively.
We know that $\mathcal{G}_4$ is positive semi-definite; therefore, either
it is positive definite or a positive integer multiplied by square of a primitive integral linear form.

Since $\mathcal{G}_4$ is positive semi-definite, it follows that its $v^2$-coefficient is positive (it can not be zero because $u \nmid \mathcal{G}_4$).
Therefore, we may rewrite (\ref{case2shape1}) as
\begin{equation} \label{case2shape2} G(u,v) = v^2 g_2(u) + v g_1(u) + g_0(u), \end{equation}
where $g_2, g_1, g_0$ are integral polynomials in $u$. Furthermore, the $u^2$-coefficient of $g_2$ is positive.
Let us denote
$$
\mathfrak{D}(u) =   \frac{g_1(u)^2 - 4 g_0(u) g_2(u)}{4 g_2(u)}.
$$
By completing the square, we obtain the expression
\begin{eqnarray}
\label{sq1'}
G(u, v) = g_2(u) \left(v +  q(u)  + \frac{h(u)}{2g_2(u)} \right)^2 - \mathfrak{D}(u),
\end{eqnarray}
where $q$ and $h$ are integral polynomials such that
$g_1(u) = 2 g_2(u) q(u) + h(u)$ and $\deg h \leq 1$. Furthermore, since $\deg g_1 \leq 3$, it follows that $\deg q \leq 1$.

We consider two cases based on whether the set
\begin{eqnarray}
\notag
\mathcal{S} = \{ \mathfrak{D}(u): u \in \bZ,  g_2(u) > 0 \}
\end{eqnarray}
is bounded or not.

Let $\xi \in \Xi$ be such that $H(\cos(\xi), \sin (\xi))$ = 0. For each $\theta \in \Theta_\xi(B)$, we let
\begin{eqnarray}
\notag
\begin{bmatrix}
 u_\theta \\
 v_\theta
\end{bmatrix}
=
A
\begin{bmatrix}
B \cos (\theta)  \\
B \sin (\theta)
\end{bmatrix}
=
B A \begin{bmatrix}
 \cos (\xi) + (\cos (\theta) - \cos(\xi)) \\
 \sin (\xi) + (\sin (\theta) - \sin(\xi))
\end{bmatrix}
=
 \begin{bmatrix}
 O(B^{\frac{1}{2}}) \\
B ( c \cos(\xi) + d \sin (\xi) )   +  O(B^{\frac{1}{2}})
\end{bmatrix},
\end{eqnarray}
where $ c \cos(\xi) + d \sin (\xi) \neq 0$.
Since $A \in \GL_2(\bZ)$, we have that $(u_\theta, v_\theta) \in  \bZ^2$ if and only if $(B \cos (\theta), B \sin (\theta) )   \in  \bZ^2$.

Let us suppose $\mathcal{S}$ is bounded.
Then
\begin{eqnarray}
\notag
G(u, v) = g_2(u) \left( v +  q(u)  + \frac{h(u)}{2g_2(u)} \right)^2 + O(1).
\end{eqnarray}
Therefore, there exists $C_0 > 0$ such that
\begin{eqnarray}
\label{sq1+}
F(B \cos (\theta), B \sin (\theta) )  =  G(u_\theta, v_\theta) > C_0  (B + B^{\frac12})^2  >  C_0 B^2,
\end{eqnarray}
for  $(u_\theta, v_\theta) \in  \bZ^2$ with $\theta \in \Theta_\xi(B)$ such that
$g_2(u_\theta) > 0$. It is easy to see that for $w \in \bZ$ satisfying $g_2(w) = 0$, if it exists, we have
\begin{eqnarray}
\label{w}
\#  \{ F(x,y) :  w = ax + by, (x, y) \in \bZ^2   \} =  \#  \{ G(w, v) :  v \in \bZ   \}   \ll 1.
\end{eqnarray}
We will be making use of Proposition \ref{2222} with
$$
\mathcal{E} = \bigcup_{ \substack{ w \in \bZ  \\ g_2(w) = 0 } }  \{ (x, y) \in \bZ^2: w =  ax + by  \}.
$$
If there exists $w \in \bZ$ such that $g_2(w) < 0$,  then we easily obtain
\begin{eqnarray}
\label{neg}
\inf_{(x, y) \in \bZ^2} F(x, y) = \inf_{(u, v) \in \bZ^2} G(u, v) = - \infty.
\end{eqnarray}

On the other hand, let us now suppose $\mathcal{S}$ is unbounded.
In this case, we have
$$
\deg (g_1^2 - 4 g_0 g_2) > 2.
$$
First we assume
\[\lim_{u \rightarrow \infty} \D(u) =  \infty.  \]
Let us recall that $q$ is an integral polynomial and set $v = - q(u)$. Then (\ref{sq1'}) becomes
$$
G(u, v) = g_2(u) \left( \frac{r(u)}{2g_2(u)} \right)^2 - \mathfrak{D}(u) = - \mathfrak{D}(u) + O(1),
$$
provided $|u|$ is sufficiently large, since $\deg r \leq 1$ and $\deg g_2 = 2$; therefore, we obtain (\ref{neg}).
Next we assume
\[\lim_{ u \rightarrow \infty} \D(u) = - \infty. \]
If $\deg (g_1^2 - 4 g_0g_2)$ is odd, then
\[\lim_{u \rightarrow -\infty} \D(u) =  \infty,\] and the same analysis from the previous case implies (\ref{neg}).
On the other hand, if $\deg (g_1^2 - 4 g_0 g_2)$ is even, then
the set $\{u \in \bR:  \D(u) > 0 \}$ is bounded. Let
$$
\mathfrak{m} = \min \left\{ 0,  \,  \max_{u \in \bR} \D(u)  \right\}.
$$
Therefore, from (\ref{sq1'}) it follows that
$$
G(u_\theta, v_\theta ) \geq  g_2(u_\theta) \left( v_\theta +  q(u_\theta)  + \frac{r(u_\theta)}{2g_2(u_\theta)} \right)^2 - \mathfrak{m},
$$
and by the same argument as above in the case $\mathcal{S}$ is bounded, we also obtain (\ref{sq1+}), (\ref{w}) and (\ref{neg}).
Therefore, this completes proof of the case when $\mathcal{G}_4$ is positive definite (In this case, $\# \Xi = 2$ which corresponds to the zeros of $H$.).
If $G_4 = a_4 L^2$ with $a_4 \in \bN$ and $L$ a primitive integral linear form such that $L \nmid H$, then we are also done if $L \nmid G_3$ (In this case, we still have $\# \Xi = 2$.). On the other hand, if $L | G_3$, then we may repeat the above argument to obtain the desired conclusion (In this case, $\# \Xi = 4$ which corresponds to the zeros of $H$ and $L$.).
\end{proof}

\begin{proposition}
\label{second}
Let
$$
F(x, y) = a_4 H (x, y)^4  + H (x, y)^2 G_{3}(x, y) + F_2 (x,y) + F_1(x,y),
$$
where $H$ is a primitive integral linear form, $a_4 \in \bN$ and $G_3$ is an integral linear form such that $H \nmid G_3$ and
no further assumptions on $F_2$ and $F_1$ (In particular, this contains $F$ as in (\ref{diffF2}) with $H$ a linear form and $r = 2$.).
Then we have either
$$
\# \mathcal{R}_F([N, 2N]) \ll \frac{N}{ \sqrt{\log N} }
\quad
\textnormal{or}
\quad
\inf_{(x, y) \in \bZ^2 } F(x, y) = - \infty.
$$
\end{proposition}
\begin{proof}
By applying a $\GL_2(\bZ)$-change of variables if necessary, we may assume $F$ is of the form
\begin{equation}
\label{case2shape}
F(x,y) = a_4 x^4  + a_3 x^2 (c x + d y) + F_2(x,y) + F_1(x,y),
\end{equation}
where  $a_3, d \in \bZ_{\neq 0}$ and $c \in \bZ$.  In particular, we have $\Xi = \{ \pi/2, 3\pi/2 \}$ in the notation of Section \ref{sec3}.

Our goal is to establish
\begin{eqnarray}
\label{somebound}
\# \mathcal{R}_F([N, 2N]) \ll  \frac{N}{ \sqrt{\log N} },
\end{eqnarray}
by showing (\ref{LOWER}) holds for each $\xi \in \Xi$ or other means,
unless
\begin{eqnarray}
\label{ineqq}
\inf_{(x, y) \in \bZ^2} F(x, y) = - \infty.
\end{eqnarray}

Let us rewrite (\ref{case2shape}) as
\begin{equation} \label{case2shape2} F(x,y) = b y^2  + y g_1(x) + g_0(x), \end{equation}
where $b \in \bZ$ and $g_1, g_0$ are integral polynomials in $x$. In particular, we have $\deg g_1 = 2$.
If $b \leq 0$, then we easily obtain (\ref{ineqq}) by setting $x = x_0$ where $g_1(x_0) \neq 0$. Therefore, we assume $b > 0$ for the remainder of the proof.

Let us denote
$$
\mathfrak{D}(x) =   \frac{g_1(x)^2 - 4 g_0(x) b}{4 b}.
$$
By completing the square, we obtain the following expression
\begin{eqnarray}
\notag
F(x,y) = b \left(y + \frac{g_1(x)}{2 b} \right)^2 - \mathfrak{D}(x).
\end{eqnarray}

We consider three cases based on $\deg \mathfrak{D}$.
We begin with the case $\deg \mathfrak{D} = 0$, i.e. $\mathfrak{D}(x) \equiv C_0$ for some $C_0 \in \bQ$.
Then it follows from Lemma \ref{single} that
$$
\# \mathcal{R}_F([N, 2N] ) \ll \sqrt{N}.
$$

Next we suppose $\deg \mathfrak{D}$ is odd. Then by setting $y = -  \lfloor \frac{g_1(x)}{2 b} \rfloor$ we have
$$
F(x,y) =  - \mathfrak{D}(x) + O(1).
$$
Therefore, it is easy to see that (\ref{ineqq}) holds in this case.

Finally, we suppose $\deg \mathfrak{D}$ is even, that is $\deg \mathfrak{D} \in \{2,4\}$. If
$$
\lim_{|x| \rightarrow \infty  } \mathfrak{D}(x) = \infty,
$$
then by setting $y = -  \lfloor \frac{g_1(x)}{2 b} \rfloor$ again we obtain (\ref{ineqq}). Therefore, we assume
$$
\lim_{|x| \rightarrow \infty  } \mathfrak{D}(x) = - \infty.
$$
First we assume $\deg \mathfrak{D} = 2$. Let us denote
$$
\mathfrak{D}(x) = \frac{ e x^2 + f x + g }{4 b}  = \frac{e}{4b} \left( x +  \frac{f}{2 e} \right)^2 - C_0,
$$
where $e \in \bZ_{< 0}$, $f, g \in \bZ$ and
$$
C_0 = \frac{ f^2 - 4 e g }{16 b e}.
$$
Therefore, we obtain
\begin{eqnarray}
\notag
F(x,y) &=& \frac{(2 b y + g_1(x) )^2}{4b}  -  \frac{( 2 e  x  + f)^2}{16 b  e }  + C_0
\\
&=&
\notag
 \frac{(2 b y + g_1(x) )^2}{4b}  +  \frac{( 2 e  x  + f)^2}{16  b  |e| }  + C_0
\\
&=&
\notag
\frac{ |e| (4 b y + 2 g_1(x) )^2 +  (2 e x + f)^2      }{16  b  |e| } +  C_0.
\end{eqnarray}
Let us define
$$
Q(u, v) =  |e| u^2  + v^2.
$$
It follows that
$$
\{ F(x,y):  (x, y) \in \bZ^2  \} \subseteq  \frac{1}{16 b |e|} \{ Q(u, v) +  16 b |e| C_0:  (u, v) \in \bZ^2   \}.
$$
Therefore, by Theorem \ref{Bar} we obtain
$$
\# \mathcal{R}_F([N, 2N]) \ll \frac{N}{ \sqrt{\log N} }.
$$

Next we assume $\deg \mathfrak{D} = 4$.
Clearly, we have
$$
F(x, y) >  \frac{1}{4b}  \max \{  (2b y + g_1(x))^2,  |\mathfrak{D}(x)|   \},
$$
for $|x|$ sufficiently large.
Since $\deg g_1 = 2$, given $y = \pm B + O(B^{3/4} )$, there exist $C_2 > C_1 > 0$ such that
$$
F(x, y) \gg
\begin{cases}
   y^2  & \mbox{if }  |x| \leq  C_1 \sqrt{B}, \\
   |\mathfrak{D}(x)| & \mbox{if } C_1 \sqrt{B} < |x| < C_2 \sqrt{B}, \\
   g_1(x)^2  & \mbox{if } |x| \geq C_2 \sqrt{B}.
\end{cases}
$$
Let $\xi \in \Xi$. Then we have
\begin{eqnarray}
\label{cossin+}
\cos (\theta) = O(B^{- \frac{1}{4}}) \quad \textnormal{and} \quad \sin (\theta) = \sin (\xi) + O(B^{-\frac{1}{4}})
\end{eqnarray}
for $\theta \in \Theta_\xi(B)$.
Therefore, it follows that there exists $C_3 > 0$ satisfying
$$
F(B \cos (\theta), B \sin (\theta)) > C_3 B^2
$$
for all $\theta \in \Theta_\xi(B)$, which implies (\ref{LOWER}). Since this estimate holds for all $\xi \in \Xi$, we obtain
from Proposition \ref{2222} that (\ref{somebound}) holds.

\end{proof}
	
\section{Quadratically composed polynomials with $H$ quadratic}
\label{sec5}
	
In this section, we treat quadratically composed polynomials with $H$ an indefinite irreducible quadratic form.
	
\begin{lemma} \label{quadpoly} Let $Q \in \bQ[x,y]$ be a quadratic polynomial. Let $Q_2$ be the degree $2$ homogeneous part of $Q$,
and suppose that $Q_2$ is an integral quadratic form with non-zero discriminant. Then there exist $q_1, q_2, q_3 \in \bQ$, depending only on $Q$, such that
\[ Q(x,y) = Q_2(x + q_1, y + q_2) + q_3.\]
\end{lemma}
	
\begin{proof} Let us write
\[Q(x,y) = a x^2 + b xy + c y^2 + d x + e y + f,\]
where $a,b,c \in \bZ$ with $b^2 - 4 ac \neq 0$ and  $d,e,f \in \bQ$.
Let $q_1, q_2 \in \bQ$ to be determined in due course.
Since
\begin{align*} Q_2(x + q_1, y + q_2)  & = a(x + q_1)^2 + b(x +q_1)(y+q_2) + c(y+ q_2)^2  \\
	& = ax^2 + bxy + cy^2 + (2a q_1 + bq_2) x + (b q_1 + 2c q_2)y + aq_1^2 + bq_1 q_2 + cq_2^2,
	\end{align*}
we set
$$
q_3 = f - (aq_1^2 + bq_1 q_2 + cq_2^2).
$$
Therefore, we obtain
\[ Q(x,y) = Q_2(x + q_1, y + q_2) + q_3\]
provided we choose $q_1$ and $q_2$ to satisfy
\[
\begin{bmatrix} d \\ e \end{bmatrix}
=
\begin{bmatrix} 2a & b \\ b & 2c \end{bmatrix} \begin{bmatrix} q_1 \\ q_2 \end{bmatrix}.
\]
Since the determinant of the matrix is non-zero, this system can be solved uniquely as follows
\[\begin{bmatrix} q_1 \\ q_2 \end{bmatrix} = \frac{-1}{b^2 - 4ac} \begin{bmatrix} 2c & -b \\ -b & 2a \end{bmatrix} \begin{bmatrix} d \\ e \end{bmatrix}.  \]
Finally, it is clear that $q_1, q_2, q_3 \in \bQ$.
\end{proof}

We now answer Question \ref{QPoon} for quadratic polynomials. In order to achieve this, we will also need the following estimate by Bernays \cite{Ber} (for the statement of the result in English, see for example \cite{BG, MO}),
which is a generalization of the well-known result by Landau \cite{La}.
\begin{theorem}[Bernays] \label{Bar}
Suppose $Q_2 \in \bZ[x,y]$ is a positive definite primitive integral quadratic form.
Then there exists $C(Q_2) > 0$, depending only on the discriminant of $Q_2$, such that
$$
\# \mathcal{R}_{Q_2}([1, N]) \sim C(Q_2) \frac{N}{\sqrt{\log N}}.
$$
\end{theorem}

\begin{corollary} \label{quadpo} Let $Q \in \bQ[x,y]$ be a quadratic polynomial. Then we have either
$$
\# \mathcal{R}_Q([N, 2N]) \ll \frac{N}{\sqrt{\log N}}
\quad
\textnormal{or}
\quad
\inf_{(x, y) \in \bZ^2} Q(x, y) = - \infty.
$$
\end{corollary}

\begin{proof}
For the purpose of this proof, we may assume without loss of generality that
$Q$ is an integral quadratic polynomial with $Q(0,0) = 0$.
Let $Q_2$ be the degree $2$ homogeneous part of $Q$.
Suppose $Q_2$ is negative semi-definite or indefinite. Then there exists
$(x_0, y_0) \in \bZ^2$ such that $Q_2(x_0, y_0) < 0$.
Since
$$
Q(N x_0, N y_0) = N^2 Q_2(x_0, y_0) + O(N),
$$
it follows that
\begin{eqnarray}
\label{infQ}
\inf_{(x, y) \in \bZ^2} Q(x, y) = - \infty.
\end{eqnarray}
Therefore, we only need to consider when $Q_2$ is positive semi-definite. Suppose $Q_2$ is not positive definite, in which case it follows\footnote{This is the same idea as Lemma \ref{possemi}.} that  $Q_2 = a L^2$ for some $a \in \bN$ and $L$ a primitive integral linear form. By applying a $\GL_2(\bZ)$-change of variables if necessary, we may assume that $Q$ is of the form
$$
Q(x, y) = a x^2 + d x + e y,
$$
where $d, e \in \bZ$. If $e \neq 0$, then we easily obtain (\ref{infQ}) by setting $x = 0$. On the other hand, if $e = 0$, then the result follows by Lemma \ref{single}.

Finally, we suppose $Q_2$ is positive definite. Then by Lemma \ref{quadpoly} we obtain
$$
Q(x,y) =  \frac{1}{d^2} Q_2( d x + m_1, d x + m_2  ) + q_3,
$$
where $d \in \bN$, $m_1, m_2 \in \bZ$ and $q_3 \in \bQ$. Therefore, it follows that
$$
\# \mathcal{R}_Q([N, 2N]) \ll \# \mathcal{R}_{Q_2}([ d^2(N - q_3) , d^2(2N - q_3)]) \ll \frac{N}{\sqrt{\log N}},
$$
where the second inequality follows from Theorem \ref{Bar}.
\end{proof}

\begin{lemma}
\label{Dirlem}
Let $H$ be an indefinite irreducible integral quadratic form. Let $\xi \in \mathbb{T}$ be
such that
$$
H(\cos(\xi), \sin (\xi)) = 0.
$$
Let $m_1, m_2 \in \bZ$ and $d \in \bN$.
Then there exists an infinite sequence of pairs of integers $ \{(x_i, y_i)\}_{i=1}^{\infty}$ such that
$$
| H(d x_i + m_1, d y_i + m_2) | \ll 1
$$
and
$$
| \arg (x_i, y_i) - \xi | \ll \frac{1}{ \sqrt{x_i^2 + y_i^2} }.
$$
\end{lemma}
\begin{proof}
We shall make use of the following version of Dirichlet's theorem on Diophantine approximation \cite{El}.
Let $m_1, m_2 \in \bZ$ and $d \in \bN$.
Given any irrational $\alpha$, there exist infinitely many $(x_i, y_i) \in \bZ^2$ satisfying
\begin{eqnarray}
\label{DD}
\left| \alpha - \frac{d y_i + m_2}{d x_i + m_1}\right| \leq  \frac{2d^2}{(d x_i + m_1)^2}.
\end{eqnarray}
We apply this result with $\alpha = \sin (\xi)/\cos (\xi)$. As a result, we obtain
$$
|\sin (\xi) (d x_i + m_1) - \cos (\xi) (d y_i + m_2)| \ll \frac{1}{ |d x_i + m_1| }.
$$
It then follows that
$$
| H( d x_i + m_1, d y_i + m_2 ) | \ll 1.
$$
It also follows from (\ref{DD}) that
$$
\left| \frac{\sin (\xi)}{\cos(\xi)} - \frac{y_i}{x_i}   \right|
\leq
\frac{2d^2}{(d x_i + m_1)^2} + \left| \frac{d y_i + m_2}{ d x_i + m_1}  - \frac{y_i}{x_i}  \right|  \ll \frac{1}{\sqrt{x_i^2 + y_i^2}}.
$$
Therefore, by the mean value theorem, we obtain
$$
| \xi - \arg (x_i, y_i) | \ll \frac{1}{ \sqrt{x_i^2 + y_i^2} }.
$$
\end{proof}

\begin{proposition}
\label{lastprop}
Let
$$
F(x, y) = a_4 H (x, y)^{2}  + H (x, y) L(x, y) + F_2 (x,y) + F_1(x,y),
$$
where $H$ is an indefinite irreducible primitive integral quadratic form, $a_4 \in \bN$, $L$ is an integral linear form and
no further assumptions on $F_2$ and $F_1$ (In particular, this contains $F$ as in (\ref{diffF2}) with $H$ an indefinite quadratic form and $r = 1$.).
Then either there exists $\lambda > 0$ such that
$$
\# \mathcal{R}_F([N, 2N]) \ll N^{1 - \lambda},
$$
or we have
$$
\inf_{(x, y) \in \bZ^2 } F(x, y) = - \infty.
$$
\end{proposition}
	
\begin{proof}
First we complete the square and  write
\[
a_4 H(x,y)^2 +   H(x,y) L(x,y) = a_4 \left(H(x,y) + \frac{  L(x,y) }{2 a_4} \right)^2 - \frac{ L(x,y)^2 }{4a_4}.\]
Then by Lemma \ref{quadpoly} we can write
\[H(x,y) + \frac{ L(x,y) }{2 a_4}  = H(x + q_1, y + q_2) + q_3, \]
for some $q_1, q_2, q_3 \in \bQ$ depending only on the coefficients of $F$. It follows that
\begin{equation}
\label{Findef1}
F(x,y) = a_4 (H(x+q_1, y+q_2) + q_3 )^2 + Q(x,y),
\end{equation}
where
$$
Q(x,y) = - \frac{L(x,y)^2}{4a_4}  + F_2(x,y) + F_1(x,y) = Q_2(x, y) + F_1(x,y).
$$
We now consider each case $ \deg Q \in \{0, 1, 2\}$ separately, where we further split the case $\deg Q = 2$ into two subcases $H \nmid Q_2$ and $H| Q_2$.

\subsection{$\deg Q = 2$ and $H \nmid Q_2$}
Let $\xi \in \Xi$. Then $Q_2(\cos (\xi), \sin (\xi))  \neq 0$ and it follows that
$$
|Q_2(\cos (\theta), \sin (\theta))| \gg 1
$$
for all $\theta \in \Theta_\xi(B)$;
therefore, we have
$$
|Q(B \cos (\theta), B \sin (\theta))| \gg |Q_2(B \cos (\theta), B \sin (\theta))| \gg B^2
$$
for $\theta \in \Theta_\xi(B)$.
Suppose 
$Q_2 (\cos (\theta), \sin (\theta)) > 0$ for all $\theta \in \Theta_\xi(B)$.
Since
$$
a_4 (H(x+q_1, y+q_2) + q_3 )^2 \geq 0,
$$
it follows that there exists $C_0 > 0$ satisfying
$$
F(B \cos (\theta), B \sin (\theta)) > C_0  B^2
$$
for all $\theta \in \Theta_\xi(B)$, which implies (\ref{LOWER}).
On the other hand, suppose 
$Q_2 (\cos (\theta), \sin (\theta)) < 0$ for all $\theta \in \Theta_\xi(B)$.
Then by Lemma \ref{Dirlem} there exist infinitely many pairs of integers
$$
(x_i, y_i) \in \bZ^2 \cap \Theta_\xi( \sqrt{x_i^2 + y_i^2} )
$$
such that
\begin{eqnarray}
\label{bound}
a_4 \left(H(x_i + q_1, y_i + q_2) + q_3\right)^2 \ll 1.
\end{eqnarray}
Since
$$
F(x_i, y_i) =  Q(x_i, y_i)  +  O(1) = Q_2(x_i, y_i)  +  O( \sqrt{x_i^2 + y_i^2}  )
$$
and $Q_2(x_i, y_i) < 0$, we obtain
\begin{eqnarray}
\label{liminf}
\inf_{(x, y) \in \bZ^2 } F(x, y) = - \infty
\end{eqnarray}
in this case. Therefore, either we have $Q_2 (\cos (\xi), \sin (\xi)) > 0$ for each $\xi \in \Xi$, in which case
by Proposition \ref{2222} there exists $\lambda > 0$ such that
$$
\# \mathcal{R}_F([N, 2N]) \ll N^{1 - \lambda},
$$
or (\ref{liminf}) holds.

\subsection{$\deg Q = 1$}
\label{case deg Q = 1}
In this case, $Q$ is an integral linear form. Let $\xi \in \Xi$.
Then $Q(\cos (\xi), \sin (\xi))  \neq 0$ and it follows that
$|Q(\cos (\theta), \sin (\theta))| \gg 1$ for all $\theta \in \Theta_\xi(B)$;
therefore, we have
$$
|Q(B \cos (\theta), B \sin (\theta))| \gg B
$$
for all $\theta \in \Theta_\xi(B)$.
Again, by Lemma \ref{DD} we find infinitely many pairs of integers
$$
(x_i, y_i) \in \bZ^2 \cap \Theta_\xi( \sqrt{x_i^2 + y_i^2} )
$$
satisfying (\ref{bound}). It follows that
$$
F(x_i, y_i) =  Q(x_i, y_i)  +  O(1).
$$
If $Q(x_i, y_i) > 0$ for $i$ sufficiently large, then by considering $\xi + \pi$ instead of $\xi$,
we may replace, without loss of generality, $(x_i, y_i)$ with $(-x_i, -y_i)$, from which we see that (\ref{liminf}) holds.

\subsection{$\deg Q = 0$}
\label{case deg Q = 0}
In this case, we have that $Q$ is the zero polynomial. Therefore, we may write
$$
F(x,y) = q_0 \left( m_0 H( d x + m_1, d y +  m_2) + m_3\right)^2,
$$
for some $q_0 \in \bQ_{>0}$, $d \in \bN$ and $m_0, m_1, m_2, m_3 \in \mathbb{Z}$.
Then it follows from Lemma \ref{single} that
$$
\# \mathcal{R}_F([N, 2N ]) \ll \sqrt{N}.
$$

\subsection{$\deg Q = 2$ and $H | Q_2$}
Finally, we consider the case $H | Q_2$. In this case, we may write
$$
Q_2 (x, y) = q_0 H ( (x + q_1) - q_1, (y+ q_2) - q_2 ) = q_0 ( H(x + q_1, y+ q_2) + q_3 ) + M(x, y) +  c_0,
$$
where $q_0, c_0 \in \mathbb{Q}$ and $M$ is either a rational linear form or the zero polynomial. 
Therefore, by completing the square, we obtain
$$
F(x, y) = a_4 \left(   H(x + q_1, y+ q_2) + q_3  - \frac{q_0}{2 a_4}  \right)^2   - \frac{q_0^2}{4 a_4}  +     M(x, y) +  c_0.
$$
If $\deg M = 1$, then we may argue in the same manner as in Section \ref{case deg Q = 1} and obtain the same conclusion.
On the other hand, if $\deg M = 0$, i.e. $M$ is the zero polynomial, then the same argument as in Section \ref{case deg Q = 0}
yields
$$
\# \mathcal{R}_F([N, 2N]) \ll \sqrt{N}.
$$
\end{proof}

\section{Proof of Proposition \ref{prime} }
\label{proofprime}

Recall that the hypotheses of Proposition \ref{prime} are $F_4$ is positive semi-definite and  $\gcd( \widetilde{F}_4, F_3, F_2, F_1) \ne 1$.
We now collect a few fundamental results from Diophantine geometry.
Our main reference for this material is \cite{Silver}.
By an affine algebraic curve, we mean a geometrically irreducible affine variety of dimension $1$.
Given an affine algebraic curve $\mathcal{C} \subseteq \bA^n$ defined over $\bQ$, by its genus $g(\mathcal{C})$ we mean the genus of any smooth projective model of $\mathcal{C}$.

\begin{theorem}[Siegel's theorem]
\label{Siegel}
Let $P \in \bQ[x,y]$ be irreducible over $\overline{\bQ}$ and
$\mathcal{C} = \{ (x,y) \in \bA^2: P(x,y) = 0  \}$.
If $g( \mathcal{C}) > 0$, then $\C(\bZ)$ is finite.
\end{theorem}

\begin{lemma}
\label{birat}
Let $P$ and $\mathcal{C}$ be as in Theorem \ref{Siegel}.
Suppose $g(\mathcal{C}) = 0$ and $\mathcal{C}(\bQ)$ is infinite.
Then $\mathcal{C}$ is birational over $\bQ$ to $\bA^1$. In particular, there exist polynomials $R_1, R_2, Q_1, Q_2 \in \bQ[t]$ such that
the two sets
$$
\mathcal{C}(\bQ)
\quad
\textnormal{and}
\quad
\left\{ \left( \frac{R_1(t)}{Q_1(t)},  \frac{R_2(t)}{Q_2(t)}  \right) : t \in \bQ, Q_1(t), Q_2(t) \neq 0    \right\}
$$
are equal up to finitely many elements.
\end{lemma}
\begin{proof}
It is known that any affine algebraic curve is birational to a smooth projective curve \cite[Theorem A.4.1.4]{Silver}.
Furthermore, since $\mathcal{C}$ is defined over $\bQ$, it follows that\footnote{One way to achieve this is to first consider the normalization of $\mathcal{C}$ in $\bA^m$ \cite[Exercise A.1.15]{Silver}
and then take its closure in $\bP^m$, which basically amounts to homogenizing the defining equations. It can be verified that both procedures are defined over $\bQ$.}
in fact $\mathcal{C}$ is birational over $\bQ$ to a smooth projective curve $\widetilde{\mathcal{C}}$ defined over $\bQ$. It is clear that $\widetilde{\mathcal{C}}(\bQ) \neq \emptyset$, because $\mathcal{C}(\bQ)$ is infinite (also see \cite[Exercise A.1.2]{Silver}). Therefore,  $\widetilde{\mathcal{C}}$ is isomorphic over $\bQ$ to $\bP^1$ \cite[Theorem A.4.3.1]{Silver}. Finally, recalling that $\bP^1$ is birational over $\bQ$ to $\bA^1$, we obtain our result.
\end{proof}

We leave the details of the following lemma to the reader.
\begin{lemma}
\label{easy}
Let $P \in \bQ[x,y]$ be irreducible over $\bQ$, but reducible over $\overline{\bQ}$, and  $\mathcal{C} = \{ (x,y) \in \bA^2: P(x,y) = 0  \}$.
Then $\mathcal{C}(\bQ)$ is finite.
\end{lemma}

We need the following explicit version of Hilbert's irreducibility theorem (see \cite[Chapter 9]{Lang} for a general introduction to this theorem).
The proof can be found in \cite[Theorem 2.5]{Cohen} and an overview of related results in \cite{CD}.
\begin{theorem}[Hilbert's irreducibility theorem]
\label{Hilb}
Given $Y \in \bQ[x, u]$ irreducible over $\bQ$, let
$$
E(B) = \#\{ t \in \bZ \cap [-B, B]: Y(x, t) \textnormal{ is reducible over } \bQ   \}.
$$
Then
$$
E(B) \ll \sqrt{B} \log B.
$$
\end{theorem}

We also need the following two simple technical lemmas. We leave the details of the first lemma to the reader.
\begin{lemma}
\label{irred}
Let $Q, R \in \bZ[t]$ be such that $\gcd(Q, R) = 1$.
Then the polynomial $x Q(t) - R(t) \in \bQ[x, t]$ is irreducible over $\bQ$.
\end{lemma}

\begin{lemma} \label{binsep}
Let $G, H \in \mathbb{Z}[x,y]$ be non-constant forms.
If $\gcd(G, H) = 1$, then
$$
\max\{  |G(x,y)|, |H(x, y)| \} \gg (x^2 + y^2)^{ \frac12 \min \{ \deg G,  \deg H \} }
$$
for all $(x,y) \in \bZ^2$.
\end{lemma}

\begin{proof}
If $G$ does not have a non-trivial real zero, then the result follows easily. Let us suppose otherwise and let
$\theta_1, \ldots, \theta_m \in \mathbb{T}$ be all the values such that $G( \cos(\theta_i), \sin(\theta_i) ) = 0$.
Since $\gcd(G, H) = 1$, it follows that $H( \cos(\theta_i), \sin(\theta_i) ) \neq 0$ for all $1 \leq i \leq m$.
Let $\varepsilon > 0$ be sufficiently small.
If $|\theta - \theta_i| < \varepsilon$ for some $1 \leq i \leq m$, then we have
$$
|H (B \cos (\theta), B \sin (\theta))| =  B^{\deg H} |H (\cos (\theta), \sin (\theta))| \gg  B^{\deg H}.
$$
Otherwise, we have
$$
|G (B \cos (\theta), B \sin (\theta))|   =  B^{\deg G} |G (\cos (\theta), \sin (\theta))| \gg  B^{\deg G}.
$$
\end{proof}

We prove Proposition \ref{prime} by considering each case
$\deg \gcd( \widetilde{F}_4, F_3, F_2, F_1) \in \{1, 2, 3, 4\}$
separately.
If $\deg \gcd( \widetilde{F}_4, F_3, F_2, F_1) = 4$, then the statement of Proposition \ref{prime} follows easily from
the definition of $\widetilde{F}_4$; therefore, we omit the details for this case. We have the following two lemmas for the remaining cases.

Given $C \in \bN$, we denote
$$
C \wp = \{ C p: p \in \wp  \},
$$
where $\wp$ is the set of prime numbers.

\begin{lemma}
\label{lemspcase}
Suppose $F_4$ is positive semi-definite and $\deg \gcd( \widetilde{F}_4, F_3, F_2, F_1) = 2$.
Let $C \in \bN$ and $D \in \bZ$. Then
$$
F(\bZ^2) \neq (C \bZ)_{\geq D}.
$$

\end{lemma}
\begin{proof}
Our assumptions imply that $F$ can be expressed in the form
\begin{equation} \label{gcd2} F(x,y) = H(x,y) G(x,y),
\end{equation}
where $H$ is a primitive integral quadratic form and $G$ is an integral quadratic polynomial.
Let $G_2$ denote the homogeneous degree $2$ part of $G$.

If $\gcd(H, G_2) = 1$, then we see from Lemma \ref{binsep} that
\[
|F(x,y)| \gg \max\{|H(x,y)|, |G(x,y)|\} =  \max\{|H(x,y)|, |G_2(x, y)| + O( \sqrt{x^2 + y^2} )\}  \gg x^2 + y^2
\]
for all $(x, y) \in \bZ^2$ such that $F(x,y) \neq 0$. Therefore, either
$$
\inf_{(x,y) \in \bZ^2} F(x,y) = -\infty
$$
holds or we are done by Proposition \ref{2222} with
$\mathcal{E} = \{ (x,y) \in \bZ^2:  F(x,y) = 0  \}$. Otherwise, we must have $\gcd(H, G_2) \ne 1$. Since $H$ is a factor of $\widetilde{F}_4$, either
$H$ is an indefinite irreducible quadratic form or $L_1 L_2$, where $L_1$ and $L_2$ are primitive
integral linear forms (possibly equal to each other).

First let us suppose $H$ is indefinite and irreducible. If the linear term of $G$ is non-zero, then Proposition \ref{lastprop} applies and we are done.
On the other hand, if the linear term of $G$ is zero, then the result follows easily from Lemma \ref{single}.
Next we suppose $H = L_1 L_2$.
If $F(x,y) = C p$ for some prime $p$, then there exist $i \in \{1,2\}$ and $d \in \bZ$ such that $d|C$ and $L_i(x,y) = d$. Without loss of generality let us say
$L_i(x,y) = ax + by$ and
$$
\{ (x,y) \in \bZ^2: L_i(x,y) = d \} = \{  ( x_0 - b z, y_0 + a z ): z \in \bZ  \}.
$$
We define
$$
F_{L_i, d}(z) = F ( x_0 - b z, y_0 + a z ).
$$
Suppose for some $i \in \{1, 2\}$ and $d|C$, we have that  $\deg F_{L_i, d} = 1$. Then
$$
\inf_{(x,y) \in \bZ^2} F(x,y) \leq  \inf_{ \substack{ (x,y) \in \bZ^2  \\ L_i(x,y) = d  } } F (x, y)
= \inf_{ z \in \bZ } F_{L_i, d}(z) = - \infty.
$$
Therefore, let us suppose $\deg F_{L_i, d} \neq 1$ for all $i \in \{1, 2\}$ and $d|C$.
In this case, by Lemma \ref{single} we obtain
$$
\# (F(\bZ^2) \cap C \wp \cap [1, N]) \leq  \sum_{i = 1, 2} \sum_{d |C} \# (F_{L_i, d}(\bZ) \cap [1, N]) \ll \sqrt{N},
$$
which completes our proof.
\end{proof}



\begin{lemma} \label{finlem} Suppose $F_4$ is positive semi-definite and  $\deg \gcd( \widetilde{F}_4, F_3, F_2, F_1) \in \{1, 3\}$. Let $C \in \bN$.
Then we have either
$$
\#  \mathcal{R}_F( C \wp  \cap  [1, N] )   \ll \sqrt{N} \log N
\quad
\textnormal{or}
\quad
\inf_{(x,y) \in \bZ^2} F(x,y) = - \infty.
$$
\end{lemma}
\begin{proof}
By applying a $\GL_2(\bZ)$-change of variables if necessary, we may rewrite $F$ as
$$
F(x,y) = (a x + b) K(x,y),
$$
where $a \in \bZ_{\neq 0}$, $b \in \bZ$ and $K$ is a degree $3$ primitive integral polynomial.
If $F(x, y) = C p$ for some prime $p$, then there exists $(d_1, d_2) \in \bZ^2$ satisfying $C = d_1 d_2$ such that $a x + b = d_2$ and $K(x, y) = d_1 p$ or vice versa.
Let
$$
Z_{d_1, d_2}^{(1)}(N) = \# \{ p \in \wp \cap [1, N/C]: ax + b = d_2,  K(x, y) = d_1 p, (x,y) \in \bZ^2   \}
$$
and
$$
Z_{d_1, d_2}^{(2)}(N) = \#  \{  p \in \wp \cap [1, N/C]: ax + b = d_1 p,  K(x, y) = d_2, (x,y) \in \bZ^2  \}.
$$
We will show that given any $i \in \{1, 2\}$ and $(d_1, d_2) \in \bZ^2$ such that $C = d_1 d_2$, we have either
$$
Z_{d_1, d_2}^{(i)}(N)  \ll \sqrt{N} \log N
\quad
\textnormal{or}
\quad
\inf_{(x,y) \in \bZ^2} F(x,y) = - \infty.
$$
As a consequence, it follows that either
$$
\# \mathcal{R}_F(  C \wp \cap [1, N]) \ll \sqrt{N} \log N
\quad
\textnormal{or}
\quad
\inf_{(x,y) \in \bZ^2} F(x,y) = - \infty
$$
holds, as desired.

Let us fix $d_1$ and $d_2$ as above.
First we deal with $Z_{d_1, d_2}^{(1)}(N)$.
Let $K'(y) = K((d_2 - b)/a , y)$. Since $Z_{d_1, d_2}^{(1)}(N) = 0$ if $(d_2 - b)/a \notin \bZ$, we assume $(d_2 - b)/a \in \bZ$.
If $\deg K'$ is odd, then we obtain  
$$
\inf_{ (x, y) \in \bZ^2}  F(x, y)  \leq   \inf_{  y \in \bZ}  d_2 K'(y) =  - \infty.
$$
On the other hand, if $\deg K'$ is even, then
$$
Z_{d_1, d_2}^{(1)}(N) \ll \# (\{  K'(y):  y \in \bZ   \} \cap [-N, N] ) \ll
\begin{cases}
\sqrt{N} & \mbox{if } \deg K' = 2, \\
1 & \mbox{if } \deg K' = 0.
\end{cases}
$$

Next we consider $Z_{d_1, d_2}^{(2)}(N)$.
Let us define
$$
P(x,y) = K(x,y) - d_2  
$$
and $\mathcal{C} = \{ (x,y) \in \bA^2: P(x,y) = 0  \}$.
We have four different cases to deal with:
\begin{enumerate}
\item $P$ is irreducible over $\bQ$, but reducible over $\overline{\bQ}$.
\item $P$ is irreducible over $\overline{\bQ}$ and $g(\mathcal{C}) > 0$.
\item $P$ is irreducible over $\overline{\bQ}$ and $g(\mathcal{C}) = 0$.
\item $P$ is reducible over $\bQ$.
\end{enumerate}

\subsection{$P$ is irreducible over $\bQ$, but reducible over $\overline{\bQ}$.}
\label{subsec0} In this case, we obtain by Lemma \ref{easy} that
$$
Z_{d_1, d_2}^{(2)}(N)  \leq \# \mathcal{C}(\bZ)  \ll 1.
$$

\subsection{$P$ is irreducible over $\overline{\bQ}$  and $g(\mathcal{C}) > 0$.}
\label{subsec1} In this case, by Siegel's theorem (Theorem \ref{Siegel}) we know that
$\mathcal{C}(\bZ)$ is finite. Therefore, it follows that
$$
Z_{d_1, d_2}^{(2)}(N) \leq \# \mathcal{C}(\bZ) \ll 1.
$$
\subsection{$P$ is irreducible over $\overline{\bQ}$  and $g(\mathcal{C}) = 0$.}
\label{subsec2}
If $\mathcal{C}(\bQ)$ is finite, then we are done. Otherwise, by Lemma \ref{birat} we have that $\mathcal{C}$ is birational to $\bA^1$.
Therefore, there exist $R_1, R_2, Q_1, Q_2 \in \mathbb{Z}[t]$ such that the two sets
$$
\mathcal{C}(\bQ)
\quad
\textnormal{and}
\quad
\left\{ \left( \frac{R_1(t)}{Q_1(t)}, \frac{R_2(t)}{Q_2(t)} \right) :  t \in \bQ, Q_1(t), Q_2(t) \neq 0   \right\}
$$
are equal up to finitely many elements. Without loss of generality we may assume that $\gcd(R_1, Q_1) = \gcd(R_2, Q_2)= 1$.
Let $\pi$ be the projection onto the first coordinate.
Clearly, if $x \in \pi(\mathcal{C}(\bZ))$ (except for at most finitely many choices of $x$), then there exists $t_0 \in \bQ$
such that
$$
x Q_1(t_0) -  R_1(t_0) = 0.
$$
We now obtain a bound for
$$
E(B) = \# \{ x \in \bZ \cap [-B, B]: x Q_1(t_0) - R_1(t_0) = 0 \textnormal{ for some } t_0 \in \bQ  \}.
$$
It is clear that for all but finitely many choices of $x_0 \in \bZ$, the degree of
$x_0 Q_1(t) - R_1(t)$ is $\max\{ \deg Q_1, \deg R_1 \}$.
Suppose $\max\{ \deg Q_1, \deg R_1 \} >  1$. Then the existence of $t_0 \in \bQ$ such that $x_0 Q_1(t_0) - R_1(t_0) = 0$ implies
$x_0 Q_1(t) - R_1(t)$ is reducible over $\bQ$. Therefore, we have
$$
E(B) \leq \# \{ x \in \bZ \cap [-B, B]:  x Q_1(t) - R_1(t) \textnormal{ is reducible over } \bQ   \} + O(1)
$$
in this case. We know by Lemma \ref{irred} that the polynomial $x Q_1(t) - R_1(t) \in \bQ[x, t]$ is irreducible over $\bQ$.
Therefore, by Hilbert's irreducibility theorem (Theorem \ref{Hilb}) it follows that
$$
\# ( \pi( \mathcal{C}(\bZ)) \cap [-B, B] )  \leq  E(B) + O(1) \ll \sqrt{B} \log B.
$$
Consequently, we obtain
\begin{eqnarray}
\notag
Z_{d_1, d_2}^{(2)}(N) &\leq&
\# \{ x \in \bZ \cap [-2N, 2N]: ax + b = d_1 p, K(x,y) = d_2, p \in \wp,  y \in \bZ \}
\\
\notag
&\leq&
\# \{ x \in \bZ \cap [-2N, 2N]: P(x,y) = 0,   y \in \bZ \}
\\
\notag
&\leq&
\# (  \pi( \mathcal{C}(\bZ))  \cap [-2N, 2N] )
\\
\notag
&\ll& \sqrt{N} \log N.
\end{eqnarray}

Let us now suppose $\max\{ \deg Q_1, \deg R_1 \} \in \{0, 1\}$. Let
$$
\frac{R_1(t)}{Q_1(t)} = \frac{ r t + s  }{ u t + v }.
$$
It is clear that if $x = R_1(t)/Q_1(t)$, then
\begin{equation}
\label{**}
t = \frac{ s -  v  x }{ u x - r }.
\end{equation}
Since $\gcd(R_1, Q_1) = 1$, it follows that $(s -  v  x)/(u x - r)$ is a non-constant rational function.
If $r = u = 0$, then it is immediate that
$$
Z_{d_1, d_2}^{(2)}(N) \ll 1.
$$
Therefore, we assume at least one of $r$ or $u$ is non-zero.
Let us denote
$$
\widetilde{R}_2(x) = (u  x - r)^m  R_2 \left(  \frac{ s - v   x }{  u   x - r } \right)
\quad
\textnormal{and}
\quad
\widetilde{Q}_2(x) = (u  x - r)^m  Q_2 \left(  \frac{ s - v    x }{ u   x - r } \right),
$$
where $m = \max\{  \deg Q_2,  \deg R_2 \}$. In particular, we have $\gcd(\widetilde{R}_2, \widetilde{Q}_2) = 1$ and
the identity
\begin{equation}
\label{res}
A(x) \widetilde{R}_2(x) + B(x)  \widetilde{Q}_2(x) = \textnormal{Res}(\widetilde{R}_2, \widetilde{Q}_2),
\end{equation}
where $A, B \in \bZ[x]$ and $\textnormal{Res}(\widetilde{R}_2, \widetilde{Q}_2) \in \bZ_{\neq 0}$ is the resultant of $\widetilde{R}_2$ and $\widetilde{Q}_2$.
If $\deg \widetilde{Q}_2 \geq 1$, then it follows that
\begin{eqnarray}
\notag
Z_{d_1, d_2}^{(2)}(N) &\leq&
\# (  \pi( \mathcal{C}(\bZ))  \cap [- 2N, 2N] )
\\
\notag
&\leq&
\# \{ x \in \bZ \cap [- 2N, 2N]  :   R_2(t)/Q_2(t) \in \bZ, t = (s - v  x )/( u x - r) \} + O(1)
\\
\notag
&\leq&
\# \{ x \in \bZ \cap [- 2N, 2N]  :   \widetilde{R}_2(x)/ \widetilde{Q}_2(x) \in \bZ  \} + O(1)
\\
\notag
&\leq&
\# \{ x \in \bZ \cap [- 2N, 2N]  :   \widetilde{Q}_2(x) | \textnormal{Res}(\widetilde{R}_2, \widetilde{Q}_2)   \} + O(1)
\\
\notag
&\ll& 1.
\end{eqnarray}

It is a basic exercise to check that $\deg \widetilde{Q}_2 = 0$ if and only if $\deg R_2 = \deg Q_2 = 0$.
Let us denote $R_2(t) = R_2$ and $Q_2(t) = Q_2$.
It is clear that if $R_2/Q_2 \notin \bZ$, then
$$
Z_{d_1, d_2}^{(2)}(N) \ll 1.
$$
On the other hand, if $R_2/Q_2 \in \bZ$, then
the two sets
$$
\pi(\mathcal{C}(\bZ))
\quad
\textnormal{and}
\quad
\{ x \in \bZ: x = (r t + s) / (u t + v), t \in \bQ, ut+ v \neq 0 \}
$$
are equal up to finitely many elements.
It is easy to see from (\ref{**}) that the latter set contains $x$, positive or negative, with arbitrarily large $|x|$,
which implies
\begin{eqnarray}
\notag
\inf_{(x, y) \in \bZ^2} F(x,y) =   \inf_{ (x, y) \in \bZ^2} (a x + b) K(x,y) \leq  \inf_{ \substack{ (x, y) \in \bZ^2 \\  K(x, y) = d_2 } } (a x + b) d_2  = - \infty.
\end{eqnarray}

\subsection{$P$ is reducible over $\bQ$}
Suppose $P$ factors over $\bQ$, which we denote by $P(x,y) = (h x + k y + \ell) G(x,y)$.
Then $K(x, y) = d_2$ if and only if $h x + k y + \ell = 0$ or $G(x,y) = 0$.
In the former case, we consider the set
$$
\{ x \in \bZ:  h x + k y + \ell = 0, y \in \bZ \},
$$
which is either empty or infinite.
If the set is empty, then it does not contribute to $Z_{d_1, d_2}^{(2)}(N)$.
On the other hand, if the set is infinite, then it contains $x$, positive or negative, with arbitrarily large $|x|$,
which implies
\begin{eqnarray}
\label{lastcase}
\inf_{(x, y) \in \bZ^2} F(x,y) =   \inf_{ (x, y) \in \bZ^2} (a x + b) K(x,y) \leq  \inf_{ \substack{ (x, y) \in \bZ^2 \\  K(x, y) = d_2 } } (a x + b) d_2  = - \infty.
\end{eqnarray}

Now we consider the latter case $G(x,y) = 0$. Let $\mathcal{C}' = \{ (x,y) \in \bA^2: G(x,y) = 0 \}$. As above, we have four different cases to deal with:
\begin{enumerate}
\item $G$ is irreducible over $\bQ$, but reducible over $\overline{\bQ}$.
\item $G$ is irreducible over ${\overline{\bQ}}$ and $g(\mathcal{C}') > 0$.
\item $G$ is irreducible over ${\overline{\bQ}}$ and $g(\mathcal{C}') = 0$.
\item $G$ is reducible over $\bQ$.
\end{enumerate}
The first three cases can be dealt with in the same way as in Sections \ref{subsec0}, \ref{subsec1} and \ref{subsec2}.
As a result, we obtain that either
$$
Z_{d_1, d_2}^{(2)}(N)  \ll \sqrt{N} \log N
\quad
\textnormal{or}
\quad
\inf_{(x,y) \in \bZ^2} F(x,y) = - \infty
$$
holds. Since the arguments are similar, we omit the details. Finally, if $G$ is reducible over $\bQ$, then
there are two linear polynomials (possibly a constant multiple of the other) $L_1$ and $L_2$ such that $G(x,y) = L_1(x,y) L_2(x,y)$.
Therefore, if $G(x,y) = 0$, then $L_1(x,y) = 0$ or $L_2(x,y) = 0$. In either case, we may argue as above to conclude
that either it does not contribute to $Z_{d_1, d_2}^{(2)}(N)$ or (\ref{lastcase}) holds.
\end{proof}

\end{document}